\documentclass[a4paper]{extarticle}
\usepackage[english]{babel}
\usepackage{graphicx}
\usepackage{framed}
\usepackage[normalem]{ulem}
\usepackage{amsmath}
\usepackage{stmaryrd}
\usepackage{amsthm}
\usepackage{amssymb}
\usepackage{amsfonts}
\usepackage{enumerate}
\usepackage{xcolor}
\usepackage{bbm}
\usepackage{comment}
\usepackage{enumitem}
\usepackage[utf8]{inputenc}
\usepackage[T1]{fontenc}
\usepackage[top=1 in,bottom=1in, left=1 in, right=1 in]{geometry}
\usepackage{hyperref}
\usepackage{authblk}

\newcommand{\R}{\mathbb{R}}
\newcommand{\N}{\mathbb{N}}

\newtheorem{theorem}{Theorem}

\newtheorem{proposition}{Proposition}
\newtheorem{assumption}{Assumption}
\newtheorem{lemme}{Lemma}

\newtheorem{remarque}{Remark}
\newtheorem{example}{Example}

\newcommand{\parent}[1]{\left( #1\right)}
\title{Malliavin derivative and sensitivity for optimal liquidation}
\date{\today}

\author[1]{Dorian Cacitti-Holland \thanks{Dorian.Cacitti-Holland@univ-lemans.fr}}
\author[1]{Laurent Denis \thanks{Laurent.Denis@univ-lemans.fr}}
\author[1]{Alexandre Popier\thanks{Alexandre.Popier@univ-lemans.fr}}
\affil[1]{\small Laboratoire Manceau de Math\'ematiques, Le Mans Universit\'e, Avenue O. Messiaen, 72085 Le Mans cedex 9, France.  }

\begin{document}

\maketitle

\begin{abstract}
  We prove that the solution of the backward stochastic differential equation with terminal singularity has a Malliavin derivative, which is the limit of the derivative of the approximating sequence. We also provide the asymptotic behavior of this derivative close to the terminal time. We apply this result to the regularity of the related partial differential equation and to the sensitivity of the liquidation problem.
\end{abstract}

\noindent
\textit{Keywords:} Backward stochastic differential equations, Malliavin calculus, singularity, optimal liquidation.

\noindent
\textit{Acknowledgements:} This work is supported by the Hubert Curien Project Procope number 50835ZC. We are grateful for this support.


\section{Introduction}

Optimal liquidation is an important and challenging topic in mathematical finance. There is a huge literature on this subject ; the reference \cite{guea:16} is interesting to have an overall view (see also the references therein). Here we consider the stochastic extension of the Almgren \& Chriss model, first exposed in \cite{almg:chri:01}. Namely for some $p > 1$, we consider the stochastic control problem to minimize the functional
$$J(t,\alpha) = \mathbb{E}^{\mathcal{F}_t} \left[\int_t^T \parent{\eta_s |\alpha_s|^p + \gamma_s |\Xi_s|^p}ds\right]$$
over all $\alpha \in \mathcal{A}(t,x)$ where $\mathcal{A}(t,x)$ is the set of admissible controls such that $\Xi$ satisfy the dynamics 
$$\Xi_s = x + \int_t^s \alpha_u du \quad t\leq s\leq T, \quad \alpha \in L^1(t,\infty) \text{ a.s.}$$
together with the terminal constraint $\Xi_T = 0$ a.s. (mandatory liquidation). This control problem and the link with backward stochastic differential equations (BSDE in abbreviate form) has been studied in \cite{anki:jean:krus:13}. It is proved that a minimizer of the functional $J$ is the process $\Xi^*$ given by 
\begin{equation} \label{eq:optimal_state}
\Xi^*_s = x \exp\parent{-\int_t^s \parent{\dfrac{Y_u}{\eta_u}}^{q-1} du}
\end{equation}
where $(Y,Z)$ is the minimal solution of the BSDE 
\begin{equation} \label{eq:liquid_BSDE}
Y_t = \xi - \int_t^T (p-1) \dfrac{|Y_s|^{q-1} Y_s}{\eta_s^{q-1}} ds + \int_t^T \gamma_s ds + \int_t^T Z_s dW_s ,
\end{equation}
with singular terminal condition $\xi = +\infty$ a.s. Here and in the rest of the paper, $q$ is the H\"older conjugate of $p$. 
Moreover the value function of this control problem is given by $v(t,x)=|x|^p Y_t$. Remark that due to the singular terminal condition, the standard notion of solution for BSDE has to be adapted (see Proposition \ref{prop:exist_sur_solution} below). 

Another way to solve this control problem is the use of the Hamilton-Jacobi-Bellman equation. In \cite{grae:hors:sere:18}, it is proved that there exists a smooth solution $v$ of the related partial differential equation (PDE in short). Both approaches are connected through the link: $Y_s = v(s,X_s)$, where $X$ is the underlying forward process. 

\bigskip 

The aim of the paper is to study the Malliavin differentiability of the solution $Y$. Our motivations are pluralist:
\begin{itemize}
\item From the theoretical point of view, this question naturally appears in our previous work \cite{caci:deni:popi:25}. We studied the continuity at time $T$ of $Y$. To obtain our result, we used the Malliavin calculus on the approximating sequence $Y^n$ (see Proposition \ref{prop:Malliavin_deriv} below). We left the Malliavin differentiability of $Y$ and the convergence of the Malliavin derivative $DY^n$ out. 
\item It is also well-known that Malliavin derivatives (and the related divergence operator and integration by parts) can be used to analyze the sensitivity of the optimal state w.r.t. the parameters (see among many others \cite{four:lasr:99,gobe:muno:05}). In our context, since the optimal state is given by \eqref{eq:optimal_state}, the Malliavin derivative of $Y$ plays a crucial role. 
\item In \cite{grae:hors:sere:18} the existence of a solution $v\in C([0,T); D(\mathcal L))$ is proved. In dimension one, it implies that $v \in C^2$ (see \cite[Remark 2.5]{grae:hors:sere:18}). Nonetheless the gradient of this solution and its behavior at time $T$ is not studied. Formally this gradient is related to the Malliavin derivative by the formula: $\partial_x v (s,X_s) \sigma(s,X_s) = D_s Y_s$, $\sigma$ being the volatility matrix of the forward process $X$. 
\end{itemize}

In this work, we prove that $Y$ has a Malliavin derivative $DY$ and that the sequence of Malliavin derivatives $DY^n$ converges to $DY$ on $[0,T)$. We also study the behavior at time $T$ of $DY$ and show that there is a singularity at time $T$: roughly speaking, $D_\theta Y_t$ tends to $+\infty$ when $t$ tends to $T$. Applications to PDE or sensitivity analysis are then provided. To the best of our knowledges, all these results are completely new.  

We are well aware that some parts of the proofs are quite heavy due to the handling of the function $y \mapsto -y |y|^{q-1}$ and its derivatives. The arguments would be much more easier in the quadratic case, that is for $q=2$. Nonetheless we also know that in practice the case $q=2$ is too restrictive. In \cite{almg:thum:haup:05} the authors analyze a large data set from the Citigroup US equity trading desks and show that $p=1.6\Leftrightarrow q=2.67$ is a good estimate of this parameter. This is the reason why we keep this general setting. 

\paragraph{Breakdown of the paper.} In Section \ref{sect:setting}, we evoke the known results concerning the BSDE \eqref{eq:liquid_BSDE}, its approximating sequence $Y^n$ and why the convergence of the sequence of derivatives $DY^n$ cannot be done with the same arguments. In few words, the proof that $Y^n$ converges to $Y$ is heavily based on the non-linearity of the BSDE \eqref{eq:liquid_BSDE}. But the equation satisfied by $DY^n$ is linear. In Section \ref{sect:malliavin_deriv}, the asymptotic expansion of $Y$ allows us to show the differentiability of $Y$ and the behavior of its derivative. The next section \ref{sect:conv_Malliavin_deriv} studies the convergence of $DY^n$ to $DY$. For the liquidation problem such convergence seems quite incidental and the reader interested by the applications can skip it. In Section \ref{sect:applications}, we explain how to apply the results to the gradient of the related PDE and to the sensitivity. In the last section, for the sake of completeness, we give the proof for the asymptotic behavior of $Y$, which is the key tool of this paper. If most of the arguments can be found in \cite{grae:popi:21}, we simplify them and give a better estimate of the parameters.

\paragraph{Notations.} In this paper we consider a deterministic time horizon $T \in \R_+^*$, a probability space $(\Omega,\mathcal{F},\mathbb{P})$, a $d$-dimensional Brownian motion $(W_t)_{0\leq t\leq T}$ defined on the probability space and $(\mathcal{F}_t)_{t\in [0,T]}$ the augmented filtration generated by $W$. $\mathbb E^{\mathcal F_t}$ is the conditional expectation knowing $\mathcal F_t$. For all $\rho \in (1,+\infty)$, we note:
\begin{itemize}
    \item $\mathbb{D}^{1,\rho}$ is the domain of the Malliavin derivative operator in $L^\rho(\Omega)$. Furthermore we note $\mathbb{D}^{1,\infty} = \underset{\rho\geq 2}{\bigcap} \mathbb{D}^{1,\rho}$. For $A \in \mathbb{D}^{1,\rho}$ we note $(D_\theta A)_{0\leq \theta \leq T}$ its Malliavin derivative and for $X$ a $\mathbb{D}^{1,\rho}$-process we note $(D_\theta X_t)_{0\leq \theta,t\leq T}$.
    \item $S^\rho(0,T)$ is the space of stochastic progressively measurable processes $(A_t)_{0\leq t\leq T}$ with values in $\R^k$ such that $$\mathbb{E} \left[ \sup_{0\leq t\leq T}|A_t|^\rho \right] < +\infty$$
    and $ S^\infty(0,T) = \bigcap_{\rho> 1}  S^\rho(0,T)$. 
    \item $H^\rho(0,T)$ is the space of stochastic progressively measurable processes $(A_t)_{0\leq t\leq T}$ with values in $\R^k$ such that $$\mathbb{E} \left[\left(\int_0^T |A_t|^2 dt\right)^{\frac{\rho}{2}} \right] < +\infty$$
    and $H^\infty(0,T) = \bigcap_{\rho> 1} H^\rho(0,T)$.
    \item Whenever the notation $T-$ appears in the definition of a process space, we mean the set of all processes whose restrictions satisfy the respective property when $T-$ is replaced by any $T-\varepsilon$, $\varepsilon > 0$. For example, $S^\rho(0,T-) = \bigcap_{\varepsilon > 0}S^\rho(0,T-\varepsilon)$. Moreover we say that a sequence $(F_n)_{n\in \N}$ converges in $S^\rho(0,T-)$ to $F \in S^\rho(0,T-)$ if for any $\varepsilon >0$, the sequence $(F_n)_{n\in \N}$ converges to $F$ in $S^\rho (0,T-\varepsilon)$.
\end{itemize}
In the rest of the paper, $C$ denotes a generic constant, which can depend on other coefficients, and may change from line to line. 

\section{Setting, known results and main contribution} \label{sect:setting}

From now on, we fix $q>1$ and $p$ still denotes its H\"older conjugate. We suppose that 
\begin{assumption} \label{assump:eta_gamma}~
    \begin{enumerate}
        \item The coefficient $\eta$ is an It\^o process:
        $$\eta_t = \eta_0 + \int_0^t b_s^\eta ds + \int_0^t \sigma_s^\eta dW_s, \quad 0\leq t\leq T,$$
        with an initial condition $\eta_0 \in \mathbb{R}$. 
        \begin{enumerate}
            \item The drift $b^\eta : \Omega\times [0,T] \longrightarrow \mathbb{R}$ and the diffusion matrix $\sigma^\eta : \Omega\times [0,T] \longrightarrow \mathbb{R}$ are progressively measurable and bounded. 
            \item There exist $\eta_*$  and $\eta^*$ in $\mathbb{R}_+^*$ such that, a.s. for any $s\in [0,T]$,
            $$0< \eta_* \leq \eta_s < \eta^*.$$
        \end{enumerate}
        \item The process $\gamma$ is a progressively measurable, non-negative and bounded: there exists $\gamma^* \in \mathbb{R}_+^*$ such that, a.s. for any $s\in [0,T]$,
        $$0\leq \gamma_s \leq \gamma^*.$$
    \end{enumerate}
\end{assumption}
\noindent In Section \ref{sect:applications}, Example \ref{example:diff_case} provides a case where Assumption \ref{assump:eta_gamma} holds. 
\begin{remarque} \label{rem:eta:gamma}
    In particular the processes $\eta$ and $\gamma$ satisfy
    $$\mathbb E \int_0^T  \left( \dfrac{1}{\eta_s^{q-1}} + \eta_s \right) ds < +\infty, \quad \mathbb E \int_0^T \gamma_s ds < +\infty.$$
\end{remarque}

Under this framework, if $\xi \in L^\rho(\Omega)$ for some $\rho> 1$, then there exists a unique solution $(Y,Z) \in S^\rho(0,T)\times H^\rho(0,T)$ to the BSDE \eqref{eq:liquid_BSDE} (see \cite[Theorem 4.2]{bria:dely:hu:03}). When $\xi=+\infty$, following \cite{anki:jean:krus:13,krus:popi:15},  we proceed by truncation and consider the following BSDE: for any $n \in \N$ 
\begin{equation} \label{edsr:tronquee}
    Y^n_t = n + \int_t^T \left( -(p-1) \dfrac{|Y^n_s|^{q-1} Y^n_s}{\eta_s^{q-1}} +  \gamma_s \right) ds - \int_t^T Z_s^n dW_s, \quad t \in [0,T]. 
\end{equation}

\begin{lemme} \label{prop:truncated:bsde}
Under Conditions {\rm \ref{assump:eta_gamma}}, the truncated BSDE (\ref{edsr:tronquee}) admits a unique solution $(Y^n,Z^n)$ in $S^\rho(0,T) \times H^\rho(0,T)$ for all $\rho\in (1,+\infty)$. Moreover the sequence $Y_n$ is non-decreasing and the process $Y^n$ is bounded from above: for $m\leq n$
 $$\forall t\in [0,T], \quad 0\leq Y^m_t \leq Y^n_t \leq  n + \mathbb E^{\mathcal{F}_t} \left[\int_t^T \gamma_s ds \right].$$
\end{lemme}
\begin{proof}
Existence and uniqueness directly follows from \cite[Theorem 4.2]{bria:dely:hu:03}. 
Now standard a priori estimate on the solution of a BSDE (see \cite[Theorem 5.30]{pard:rasc:14}) and the comparison theorem (see \cite[Theorem 5.33]{pard:rasc:14}) imply that $Y^n$ is non-negative and the desired estimation. 
This achieves the proof of this lemma. 
\end{proof}

Since $Y^n_t$ is a non-decreasing sequence, its limit $Y_t$ exists. However the upper estimate on $Y^n$ is not sufficient to ensure that $Y_t$ is finite. According to \cite{krus:popi:15,caci:deni:popi:25}, we have:

\begin{proposition} \label{prop:exist_sur_solution}
Under Conditions {\rm \ref{assump:eta_gamma}}, the sequence $(Y^n,Z^n)$ converges to $(Y,Z)$ in $S^\infty (0,T-) \times H^\infty(0,T-)$. The limit $(Y,Z)$ is the minimal supersolution to the BSDE \eqref{eq:liquid_BSDE} on $[0,T[$ in the sense that:
\begin{enumerate}
    \item The couple $(Y,Z)$ belongs to $S^\infty (0,T-) \times H^\infty(0,T-)$ and the process $Y$ is non negative.
    \item For all $0\leq s\leq t<T$, $$Y_s = Y_t + \int_s^t \left(  -(p-1) \dfrac{|Y_r|^{q-1} Y_r}{\eta_s^{q-1}} +  \gamma_r \right) dr 
    - \int_s^t Z_r dW_r.$$
    \item The process $Y$ satisfies a.s. 
    $\lim_{t\to T } Y_t =+\infty.$
    \item The process $(Y,Z)$ is minimal: if $(\tilde{Y},\tilde Z)$ satisfies the three previous points, then a.s. for any $t$, $Y_t \leq \tilde Y_t$.
\end{enumerate}
Moreover for any $t \in [0,T]$ and $n \geq 1$:
\begin{align} \label{eq:a_priori_estimate_Y_L}
Y^{n}_t \leq & \frac{1}{(T-t+\frac{1}{n^{q-1}})^{p}} \left\{ \dfrac{1}{n^{q-1}} + \mathbb E^{\mathcal{F}_t} \left[ \int_t^{T} \left( \eta_s + (T-s+1)^{p} \gamma_s \right) ds\right] \right\}.
\end{align}
\end{proposition}
As a consequence, the process $Y$ satisfies: for all $0\leq t < T$
\begin{equation}\label{eq:a_priori_estimate} 
0\leq Y^n_t \leq Y_t \leq \frac{1}{(T-t)^{p}}    \mathbb E^{\mathcal{F}_t} \left[ \int_t^{T} \left( \eta_s + (T-s+1)^{p} \gamma_s \right) ds \right].
\end{equation}

We show the Malliavin differentiability of the couple $(Y^n,Z^n)$, due to \cite[Theorem 5.1 and Application 6.1]{mastrolia:2015}. We assume that: 
\begin{assumption} \label{assumption:Malliavin:b:eta:gamma}
    The processes $b^\eta,\eta,\gamma$ are in valued in $\mathbb{D}^{1,2}$, their Malliavin derivatives $Db^\eta, D\eta, D\gamma$ admit progressively measurable versions which are in $L^2(\Omega \times [0,T] \times [0,T])$.
\end{assumption}

\begin{proposition} \label{prop:Malliavin_deriv}
    Under Conditions \ref{assump:eta_gamma} and \ref{assumption:Malliavin:b:eta:gamma}, the solution $(Y^n, Z^n)$ of the truncated BSDE (\ref{edsr:tronquee}) is in $L^2([0,T],\mathbb{D}^{1,2}\times \mathbb{D}^{1,2})$. Moreover for all $0\leq t < \theta \leq T$, $D_\theta Y^n_t = 0$, $D_\theta Z^n_t = 0$ and, for all $0\leq \theta \leq t\leq T$,
    \begin{align} \label{eq:BSDE_Malliavin_deriv}
        D_\theta Y^n_t & = \int_t^T \left(  -(p-1) \dfrac{|Y^n_r|^{q-1} }{\eta_r^{q-1}}  D_\theta Y^n_r + \dfrac{|Y^n_r|^{q-1}Y^n_r }{\eta_r^{q}} D_\theta \eta_r+ D_\theta \gamma_r \right) dr- \int_t^T D_\theta Z^n_s dW_s. 
    \end{align}
\end{proposition}

\begin{proof}
    Under our setting, we can directly apply \cite[Proposition 6]{caci:deni:popi:25}.
\end{proof}

In our setting, the approximating sequence $Y^n$ has a Malliavin derivative $D Y^n$, which satisfies the linear BSDE given in Proposition \ref{prop:Malliavin_deriv}. Very natural questions are: does the minimal solution $Y$ have also a Malliavin derivative $DY$ and do we have the convergence of $DY^n$ to $DY$? Evoke that to obtain the convergence of $Y^n$, the a priori estimate \eqref{eq:a_priori_estimate_Y_L} together with the monotonicity of the sequence $Y^n$ are crucial. We do not have similar tools for the Malliavin derivative. And we derive the a priori estimate using the non-linearity of the BSDE \eqref{eq:liquid_BSDE}. Hence it is not possible to directly pass to the limit in the linear BSDE \eqref{eq:BSDE_Malliavin_deriv}.

Formally if we pass to the limit in \eqref{eq:BSDE_Malliavin_deriv}, we should have a linear BSDE of the form:
\begin{align*} 
    U_t & = \int_t^T \left(  -(p-1) \dfrac{|Y_r|^{q-1} }{\eta_r^{q-1}}  U_r + \dfrac{|Y_r|^{q-1}Y_r }{\eta_r^{q}} D_\theta \eta_r+ D_\theta \gamma_r \right) dr- \int_t^T V_s dW_s
\end{align*}
with a singular generator since a.s.  
$$\int_0^T \dfrac{|Y_r|^{q-1} }{\eta_r^{q-1}} dr = +\infty.$$
This property comes from the liquidation condition $\Xi_T^* = 0$ a.s. and Equation \eqref{eq:optimal_state}.
In \cite{jean:mast:poss:15,jean:reve:14}, such linear BSDEs with singular generator are studied. 
Nonetheless to apply the results of these papers, the process $\dfrac{|Y|^{q-1}Y }{\eta^{q}} D_\theta \eta+ D_\theta \gamma$ should be bounded. In general this property does not hold. 

\medskip 

However in the liquidation problem, that is for the BSDE \eqref{eq:liquid_BSDE} when $\xi=+\infty$ a.s., we can prove the Malliavin differentiability of $Y$. Let us evoke the results of \cite{grae:popi:21} thanks to Assumption \ref{assump:eta_gamma}. The minimal solution $Y$ of \eqref{eq:liquid_BSDE} can be written: for all $0 \leq t < T$
\begin{equation}\label{eq:asymp_Y}
    Y_t = \dfrac{\eta_t}{(T-t)^{p-1}} + \dfrac{1}{(T-t)^{p}} H_t. 
\end{equation}
The process $H$ is the unique solution of the BSDE:
\begin{equation}\label{eq:BSDE_H}
H_t = \int_t^T F(s,H_s) ds - \int_t^T Z^H_s dW_s, \quad 0 \leq t\leq T
\end{equation}
with generator $F$ given by: for all $0 \leq t < T$
\begin{align} \nonumber
F(t,h) & = \left[ (T-t) b^\eta_t  + (T-t)^{p} \gamma_t \right] \\ \label{eq:gene_F} 
& \quad - (p-1)\eta_t \left[ \left( 1 + \dfrac{1}{\eta_t (T-t)} h\right) \left| 1 + \dfrac{1}{\eta_t (T-t)} h\right|^{q-1} -1 - q\dfrac{1}{\eta_t (T-t)} h \right] .
\end{align}
This generator $F$ is singular at time $T$, in the sense that for $h \neq 0$,  a.s.
$$\int_0^T |F(t,h)|dt = + \infty.$$
Therefore existence and uniqueness of $H$ needs to be explained. Here we prove that there exist two constants $\delta > 0$ and $R \geq 0$, depending only on $q$, $T$, $\eta$ and $\gamma$, such that on the time interval  $[0,T]$, $H$
satisfies a.s. for any $t\in [T-\delta,T]$, $|H_t| \leq R (T-t)^2$, $H$ is bounded on $[0,T]$, and the solution $(H,Z^H)$ is obtained by a Picard iteration argument in the space of adapted processes
$$\mathcal H^\delta= \{ H \in L^\infty(\Omega; C([T-\delta,T];\R)), \ \|H\|_{\mathcal H^\delta} < +\infty\}$$
endowed with the weighted norm
$$\|H\|_{\mathcal H^\delta} = \left\| \sup_{t\in [T-\delta,T]} (T-t)^{-2} H_t \right\|.$$
See the appendix for more details. Let us give the value of the constants $R$ and $\delta$:
\begin{equation}\label{eq:def_constants} 
R =  \|b^\eta\|_\infty  + \dfrac{2}{p+1} \gamma^\star ,\quad L= \dfrac{qR}{\eta_\star} 2^{|q-2|},\quad \delta = \min \left( 1, T, \dfrac{1}{2L}, \dfrac{\eta_\star}{2R}\right).
\end{equation}

The main result of this paper is:

\begin{theorem} \label{thm:main_result}
    The process $Y$ is valued in $\mathbb{D}^{1,2}$ and, for any $0\leq \theta \leq t < T$,
    \begin{equation} \label{eq:deriv_Y_explicit}
        D_\theta Y_t = \dfrac{1}{(T-t)^{p-1}} D_\theta \eta_t + \dfrac{1}{(T-t)^{p}} D_\theta H_t
    \end{equation}
    and if for some $\varrho > 1$,
    $$\sup_{\theta \in [0,T]} \mathbb E \left[ |D_\theta \eta_T|^{\varrho} +  \int_0^T \left( |D_\theta b^\eta_s|^{\varrho} + |D_\theta \gamma_s|^{\varrho} + |D_\theta \eta_s|^{\varrho} \right) ds\right] < +\infty$$
    then for $1< \ell < \varrho$,
    $$ \lim_{n\to +\infty} \sup_{\theta \in [0,T]} \mathbb E \left[ \sup_{t\in [0,T]} (T-t)^{\ell p} |D_\theta Y_t -D_\theta Y^n_t |^\ell \right] = 0.$$
    In particular for any $0 \leq \tau < T$
    $$\lim_{n\to +\infty} \sup_{\theta \in [0,T]} \mathbb E \left[ \sup_{t\in [0 ,\tau]} |D_\theta Y_t -D_\theta Y^n_t |^\ell \right] = 0.$$
\end{theorem}

\section{Malliavin derivative for the minimal solution and its behaviour at time $T$} \label{sect:malliavin_deriv}

Evoke that the generator $F$ is given by \eqref{eq:gene_F} and we define on the set $[0,T) \times \mathbb R \times [\eta_\star,\eta^\star]$ the function $G$ by:
\begin{align*}
G(t,h,\eta)& = (p-1)\eta \left[\left( 1 + \dfrac{1}{\eta (T-t)} h\right) \left| 1 + \dfrac{1}{\eta (T-t)} h\right|^{q-1} -1 - q\dfrac{1}{\eta (T-t)} h \right]. 
\end{align*}
Thus we have
$$F(t,h) = (T-t) b_t^\eta + (T-t)^p \gamma_t - G(t,h,\eta_t), \quad 0\leq t<T, h\in \R.$$
On the set $[0,T) \times \mathbb R \times [\eta_\star,\eta^\star]$ we have:
\begin{align}
    \label{eq:deriv_G_h_0} \dfrac{\partial G}{\partial h}(t,h,\eta)& = \dfrac{p}{ (T-t)} \left( \left| 1 + \dfrac{1}{\eta (T-t)} h\right|^{q-1} -  1 \right)     \\  \nonumber
    \dfrac{\partial G}{\partial \eta}(t,h,\eta)& = (p-1)\left[\left( 1 + \dfrac{1}{\eta (T-t)} h\right) \left| 1 + \dfrac{1}{\eta (T-t)} h\right|^{q-1} -1 - q\dfrac{1}{\eta (T-t)} h \right] \\ \nonumber
    &  -\dfrac{ph}{\eta (T-t)} \left( \left| 1 + \dfrac{1}{\eta (T-t)} h\right|^{q-1} -  1 \right)  \\ \label{eq:deriv_G_eta_0}
    & =  - \left( 1 + \dfrac{1}{\eta (T-t)} h\right)\left| 1 + \dfrac{1}{\eta (T-t)} h\right|^{q-1} + 1 + p \left( \left| 1 + \dfrac{1}{\eta (T-t)} h\right|^{q-1} -  1 \right).
\end{align}
In the case where $q = p = 2$, we get the easier expression
$$G(t,h,\eta) = \dfrac{h^2}{\eta (T-t)^2},\quad \dfrac{\partial G}{\partial h}(t,h,\eta) =  \dfrac{2h}{\eta (T-t)^2} ,\quad \dfrac{\partial G}{\partial \eta}(t,h,\eta) = - \dfrac{h^2}{\eta^2 (T-t)^2} .$$
A key point in the proof of \cite[Theorem 23]{grae:popi:21} is a.s. for any $T-\delta \leq t  \leq T$, $|H_t| \leq R(T-t)^2$. Therefore these derivatives are bounded when we replace $(h,\eta)$ by the processes $(H_t,\eta_t)$ for $t\in [T-\delta,T]$. 

From time to time, we use the next representation of $G$ and their derivatives:
\begin{align} \nonumber
G(t,h,\eta)&  = \dfrac{q h^2}{\eta (T-t)^2} \int_0^1\left| 1 + a\dfrac{1}{\eta (T-t)} h\right|^{q-2} \mbox{sign } \left(1 + a\dfrac{1}{\eta (T-t)} h\right) (1-a) da\\ \label{eq:deriv_G_h}
 \dfrac{\partial G}{\partial h}(t,h,\eta)&= \dfrac{qh}{\eta (T-t)^2} \int_0^1\left| 1 + a\dfrac{1}{\eta (T-t)} h\right|^{q-2} \mbox{sign } \left(1 + a\dfrac{1}{\eta (T-t)} h\right)  da \\ \label{eq:deriv_G_eta}
  \dfrac{\partial G}{\partial \eta}(t,h,\eta)& =  - \dfrac{qh^2}{\eta^2 (T-t)^2} \int_0^1\left| 1 + a\dfrac{1}{\eta (T-t)} h\right|^{q-2} \mbox{sign } \left(1 + a\dfrac{1}{\eta (T-t)} h\right) a da. 
\end{align}
These representations may be not well-defined when $q < 2$ if $1 + a \dfrac{1}{\eta (T-t)} h = 0$. However in the construction of the solution $H$, $\delta$ is chosen such that 
\begin{equation} \label{eq:choix:delta}
    \dfrac{|H_t|}{\eta_t (T-t)} \leq \dfrac{1}{2}.
\end{equation}
 Hence we can use these versions with $h=H_t$ on the whole time interval $[T-\delta,T]$. 

\begin{lemme} \label{lem:mallia_deriv_H}
    Under the previous assumptions, the process $H$ has a Malliavin derivative, such that $D_\theta H_t$ is equal to $0$ if $t < \theta$ and for any $0 \leq \theta \leq t \leq T$:
    \begin{equation} \label{eq:BSDE_Malliavin_deriv_H}
    D_\theta H_t = \mathbb E^{\mathcal F_t}  \int_t^T  \left( (T-s) D_\theta b^\eta_s  + (T-s)^p D_\theta \gamma_s - \dfrac{\partial G}{\partial \eta}(s,H_s,\eta_s) D_\theta \eta_s - \dfrac{\partial G}{\partial h}(s,H_s,\eta_s)  D_\theta H_s \right) ds  .
    \end{equation}
\end{lemme}
\begin{proof}
    It is an adaptation of the proof of \cite[Proposition 5.4]{elka:peng:quen:97}. The solution $(H,Z^H)$ is the limit in $\mathcal H^\delta$ of the sequence $(H^k,Z^{H,k})$ unique solution in $\mathcal H^\delta$ of 
    $$H^k_t = \int_t^T F(s,H^{k-1}_s) ds - \int_t^T Z^{H,k}_s dW_s$$
    with $(H^0,Z^{H,0}) = (0,0)$. Moreover for any $k$ and $t \in [T-\delta,T]$, $|H^k_t| \leq R (T-t)^2$ and $|H^k_t|\leq C$. 
    By recursion we prove that $H^k$ has a Malliavin derivative $D H^k$ such that for $\theta \leq t \leq T$ and $t \geq T-\delta$:
\begin{align*}
D_\theta H^k_t & = \mathbb E^{\mathcal F_t}  \int_t^T  \left( (T-s) D_\theta b^\eta_s  + (T-s)^p D_\theta \gamma_s - \dfrac{\partial G}{\partial \eta}(s,H^{k-1}_s,\eta_s) D_\theta \eta_s - \dfrac{\partial G}{\partial h}(s,H^{k-1}_s,\eta_s)  D_\theta H^{k-1}_s \right)ds. 
\end{align*}
The processes $\dfrac{\partial G}{\partial \eta}(s,H^k_s,\eta_s)$ and $\dfrac{\partial G}{\partial h}(s,H^k_s,\eta_s)$ are bounded, uniformly w.r.t. $k$. Indeed using the representations \eqref{eq:deriv_G_h} and \eqref{eq:deriv_G_eta}  we obtain that for $T-\delta \leq s \leq T$
\begin{align*}
    &\left| \dfrac{\partial G}{\partial h}(s,H^k_s,\eta_s) \right| \leq  \dfrac{q|H_s|^k}{\eta_s (T-s)^2} \int_0^1\left| 1 + a\dfrac{1}{\eta_s (T-s)} H_s^k\right|^{q-2}  da\\
    & \quad
     \leq  \dfrac{qR}{\eta_\star} \int_0^1\left| 1 + a\dfrac{1}{\eta_s (T-s)} H_s^k\right|^{q-2} da
     \leq  \dfrac{qR}{\eta_\star} \int_0^1\left( 1 + a\dfrac{1}{\eta_s (T-s)} |H_s^k|\right)^{q-2} da \\ &\quad  \leq  \dfrac{qR}{\eta_\star} \int_0^1\left( 1 + \dfrac{a}{2}\right)^{q-2} da
\leq  \dfrac{q R}{\eta_\star} 2^{|q-2|},
 \end{align*}
 where we used that $|H^k_t| \leq R (T-t)^2$ and that Equation \eqref{eq:choix:delta} also holds for $H^k$ from the choice of $\delta$ given by \eqref{eq:def_constants}. Similarly
\begin{align*}
& \left| \dfrac{\partial G}{\partial \eta}(s,H^k_s,\eta_s) \right|  \leq  \dfrac{q(H^k_s)^2}{\eta_s^2 (T-s)^2} \int_0^1\left| 1 + a\dfrac{1}{\eta_s (T-s)} H_s^k\right|^{q-2} a da
    \\ &\quad  \leq  \dfrac{qR^2 (T-s)^2}{\eta_\star^2} \int_0^1\left( 1 + \dfrac{a}{2}\right)^{q-2} a da \leq  \dfrac{qR^2 (T-s)^2}{\eta_\star^2} 2^{|q-2|}.
\end{align*}
Hence using classical a priori estimates for BSDE (see \cite[Section 5.3.1]{pard:rasc:14}), we prove that the sequence $D_\theta H^k$ converges to the solution of \eqref{eq:BSDE_Malliavin_deriv_H} on $[T-\delta, T]$. Moreover we can also use classical arguments on $[0,T-\delta]$ (\cite{elka:peng:quen:97} or \cite{mastrolia:2015}) thanks to the expression of $H$ on $[0,T-\delta]$ given by \eqref{eq:asymp_Y} (see end of the proof of Proposition \ref{prop:existence_H} for the properties of $F$ on $[0,T-\delta]$). 
\end{proof}

Note that for $q=2$, the BSDE \eqref{eq:BSDE_Malliavin_deriv_H} becomes
    \begin{equation*} 
    D_\theta H_t =\mathbb E^{\mathcal F_t} \int_t^T  \left[ (T-s) D_\theta b^\eta_s  + (T-s)^2 D_\theta \gamma_s + \dfrac{H_s^2}{\eta_s^2 (T-s)^2} D_\theta \eta_s - \dfrac{2 H_s}{\eta_s(T-s)^2} D_\theta H_s \right]ds.
    \end{equation*}
\begin{lemme}
There exists $C$ such that the Malliavin derivative of $H$ satisfies: a.s. for any $0\leq \theta \leq t \leq T$
\begin{equation}\label{eq:estim_Mallia_deriv_H}  
  |D_\theta H_t| \leq C (T-t) \mathbb E^{\mathcal{F}_t}\left[\int_t^T  \left( | D_\theta b^\eta_s|  +  |D_\theta \gamma_s| + | D_\theta \eta_s | \right)  ds \right].
  \end{equation}
\end{lemme}
\begin{proof}
    Evoke that $|H_t| \leq R(T-t)^2$ and remark that, by solution of the linear BSDE \eqref{eq:BSDE_Malliavin_deriv_H},
  \begin{equation}  \label{eq:deriv_H_explicit}
  D_\theta H_t = \mathbb E^{\mathcal{F}_t} \left [\int_t^T  \left[ (T-s) D_\theta b^\eta_s  + (T-s)^{p} D_\theta \gamma_s -\dfrac{\partial G}{\partial \eta}(s,H_s,\eta_s) D_\theta \eta_s \right]  \Gamma_{t,s}  ds \right] 
  \end{equation}  
  where 
  \begin{equation} \label{eq:def_Gamma}
  \Gamma_{t,s} = \exp \left( -\int_t^s \dfrac{\partial G}{\partial h}(u,H_u,\eta_u)  du \right).
  \end{equation}
We already know that on $[T-\delta,T]$:
$$
 \left|  \dfrac{\partial G}{\partial h}(s,H_s,\eta_s) \right| \leq  \dfrac{q R}{\eta_\star} 2^{|q-2|} , \quad \left|   \dfrac{\partial G}{\partial \eta}(s,H_s,\eta_s) \right| \leq    \dfrac{qR^2}{\eta_\star^2}2^{|q-2|} (T-s)^2 . 
$$
If $ 0\leq \theta \leq T$ and $(T-\delta)\vee \theta \leq t \leq T$, the result directly follows from \eqref{eq:deriv_H_explicit}. If $0\leq \theta \leq t \leq T-\delta$, 
$$D_\theta H_t = \mathbb E^{\mathcal{F}_t} \left [  D_\theta H_{T-\delta} \Gamma_{t,T-\delta} \right] + \mathbb E^{\mathcal{F}_t} \left [  \int_t^{T-\delta}  \left[ (T-s) D_\theta b^\eta_s  + (T-s)^{p} D_\theta \gamma_s -\dfrac{\partial G}{\partial \eta}(s,H_s,\eta_s) D_\theta \eta_s \right]  \Gamma_{t,s}  ds \right]$$
and $\left|\dfrac{\partial G}{\partial h}(\cdot,H,\eta) \right|$ and $\left|\dfrac{\partial G}{\partial \eta}(\cdot,H,\eta) \right|$ are bounded processes on $[0,T-\delta]$. Hence there exists a constant $C$ such that 
\begin{align*}  
& | D_\theta H_t | \leq C \mathbb E^{\mathcal{F}_t} \left [  |D_\theta H_{T-\delta}| \right] + \mathbb E^{\mathcal{F}_t} \left[ \int_t^{T-\delta}  \left( |D_\theta b^\eta_s|  + |D_\theta \gamma_s| +| D_\theta \eta_s| \right) ds \right]  \\
 &\quad \leq C \mathbb E^{\mathcal{F}_t} \left [ C \delta  \int_{T-\delta}^T  \left( | D_\theta b^\eta_s|  +  |D_\theta \gamma_s| + | D_\theta \eta_s | \right)  ds +  \int_t^{T-\delta}  \left( |D_\theta b^\eta_s|  + |D_\theta \gamma_s| +| D_\theta \eta_s| \right) ds \right]  \\
 &\quad \leq C\mathbb E^{\mathcal{F}_t} \left[ \int_t^{T} \left( |D_\theta b^\eta_s|  + |D_\theta \gamma_s| +| D_\theta \eta_s| \right) ds \right]  = \dfrac{C}{\delta} (T-t)\mathbb E^{\mathcal{F}_t} \left [   \int_t^{T}  \left( |D_\theta b^\eta_s|  + |D_\theta \gamma_s| +| D_\theta \eta_s| \right) ds \right]. 
  \end{align*}  
\end{proof}

From this estimate, we have
$$\mathbb E \int_0^T \int_0^T |D_\theta H_t|^2 d\theta dt < +\infty$$
and from Equation \eqref{eq:asymp_Y}, we deduce the beginning of Theorem \ref{thm:main_result} with Equation \eqref{eq:deriv_Y_explicit}. 
From Equation \eqref{eq:estim_Mallia_deriv_H} and Assumption \ref{assumption:Malliavin:b:eta:gamma}, we deduce that:
\begin{equation*}
  |D_\theta H_t| \leq C (T-t)^{3/2} \mathbb E^{\mathcal{F}_t} \left[ \left( \int_t^T  \left( | D_\theta b^\eta_s|^2  +  |D_\theta \gamma_s|^2 + | D_\theta \eta_s |^2 \right)  ds \right)^{1/2} \right].
  \end{equation*}
Thus close to $T$, the leading term in $D_\theta Y_t $ is $\dfrac{D_\theta \eta_t}{(T-t)^{p-1}}$, and if $\lim_{t\to T} D_\theta \eta_t = D_\theta \eta_T$ a.s. (for example if $\eta_t = \eta(t,X_t)$ with $X$ solution of SDE \eqref{eq:sde}, see Section \ref{ssect:gradient}) then, on the set $\{ D_\theta \eta_T \neq 0\}$,
$$\lim_{t\to T} |D_\theta Y_t | = +\infty.$$
Therefore there is a singularity at time $T$ for the Malliavin derivative, even in this quite simple case where $\xi = +\infty$ a.s. And we expect that a similar singular behavior holds in the general setting studied in \cite{krus:popi:15,caci:deni:popi:25} when $\xi$ can be equal to $+\infty$ with a positive probability. 

\begin{remarque}
    However if $\eta$ is deterministic, then $D_\theta \eta = 0$ and $D_\theta Y_t =\dfrac{1}{(T-t)^{p}} D_\theta H_t$
    and from \eqref{eq:deriv_H_explicit}
    $$  |D_\theta H_t | \leq \mathbb E^{\mathcal{F}_t} \left [\int_t^T (T-s)^{p} |D_\theta \gamma_s | \Gamma_{t,s}  ds \right] \leq C (T-t)^p  \mathbb E^{\mathcal{F}_t} \left [\int_t^T  |D_\theta \gamma_s |  ds \right] .$$
    Then a.s. 
    $$\lim_{t\to T} |D_\theta Y_t | = 0.$$
    In this case we do not have singularity of the derivative.
\end{remarque}

\section{Convergence of the Malliavin derivatives} \label{sect:conv_Malliavin_deriv}

Evoke that $Y^n_t$ converges a.s. to $Y_t$ given by \eqref{eq:asymp_Y}. Moreover $Y^n$ and $Y$ belong to $\mathbb D^{1,2}$ (Proposition \ref{prop:Malliavin_deriv} and first part of Theorem \ref{thm:main_result}). 
The goal of this section is to that the sequence $D Y^n$ converges to $DY$.  

Let us define the process $\mathcal H^n$ by:
$$\mathcal{H}^n_t = \left(T-t+ \left(\frac{\eta^\star}{n}\right)^{q-1}\right)^{p} Y_t^n - \left(T-t+ \left(\frac{\eta^\star}{n}\right)^{q-1}\right)\eta_t, \quad 0\leq t\leq T,$$
to get
\begin{equation} \label{eq:definition:Hn}
Y^n_t = \dfrac{1}{\left(T-t+ \left(\frac{\eta^\star}{n}\right)^{q-1}\right)^{p-1}} \eta_t + \dfrac{1}{\left(T-t+ \left(\frac{\eta^\star}{n}\right)^{q-1}\right)^{p}} \mathcal H^n_t.
\end{equation}
Thus 
$$\mathcal{H}^n_T = \left( \left(\frac{\eta^\star}{n}\right)^{q-1}\right)^{p} Y_T^n - \left(\frac{\eta^\star}{n}\right)^{q-1}\eta_T = \left(\frac{\eta^\star}{n}\right)^{q} Y_T^n - \left(\frac{\eta^\star}{n}\right)^{q-1}\eta_T=\left(\frac{\eta^\star}{n}\right)^{q} n - \left(\frac{\eta^\star}{n}\right)^{q} \frac{n}{\eta^\star} \eta_T.$$
The process $\mathcal H^n$ is the solution of the BSDE 
\begin{equation} \label{eq:BSDE_H_n}
\mathcal H^n_t =  (\eta^\star)^q\left( 1 - \dfrac{\eta_T}{\eta^\star}\right) \dfrac{1}{n^{q-1}} + \int_t^T F \left( s- \left(\dfrac{\eta^\star}{n}\right)^{q-1} , \mathcal H^n_s \right) ds - \int_t^T \mathcal Z^n_s dW_s.
\end{equation}
Note that the terminal value is non-negative:
$$\mathcal H^n_T = (\eta^\star)^q\left( 1 - \dfrac{\eta_T}{\eta^\star}\right) \dfrac{1}{n^{q-1}} \geq 0.$$

Let us start with an estimate on $\mathcal H^n$. With an abuse of notations, we still denote the constant $\delta$ in the next statement. Indeed we can always take the minimum between the constant $\delta$ coming from the existence of the process $H$ and the constant $\delta$ coming from this lemma.
\begin{lemme} \label{lem:bounds_on_H^n}
    There exist $\delta > 0$, $n_0 \in \mathbb N$ and two positive constants $C_1$ and $C_2$ s.t. for any $n\geq n_0$ and $T-\delta \leq t \leq T$:
$$-C_1 \left( T-t + \left( \frac{\eta^\star}{n}\right)^{q-1}\right)^2 \leq \mathcal H^n_t \leq C_2 \left( T-t + \left( \frac{\eta^\star}{n}\right)^{q-1}\right).  $$
Moreover there exists $C$ such that for any $n$, on $[0,T-\delta]$, $|\mathcal H^n_t| \leq C$. 
\end{lemme}
\begin{proof}
    We already know that 
    $$0\leq \mathcal H^n_T = (\eta^\star)^q\left( 1 - \dfrac{\eta_T}{\eta^\star}\right) \dfrac{1}{n^{q-1}} \leq\dfrac{(\eta^\star)^q}{n^{q-1}} .$$
    Hence the estimate holds at time $T$ for any $C_1 \geq 0$ and $C_2 \geq \eta^\star$.

    Now if $V_t = C_2 \left( T-t + \left( \frac{\eta^\star}{n}\right)^{q-1}\right)$, then $-dV_t = C_2 dt$ and 
    \begin{align*}
    F\left( t- \left( \frac{\eta^\star}{n}\right)^{q-1} , V_t  \right) & = \left[ \left(T-t+\left( \frac{\eta^\star}{n}\right)^{q-1} \right) b^\eta_t  + \left(T-t+\left( \frac{\eta^\star}{n}\right)^{q-1} \right)^{p} \gamma_t \right]\\
    & \qquad - (p-1) \eta_t \left[ \left( 1 + \dfrac{C_2}{\eta_t } \right)^q - 1 - \dfrac{qC_2}{\eta_t} \right].
       \end{align*}
Hence $F \left( t- \left(\dfrac{\eta^\star}{n}\right)^{q-1} , V_t  \right) \leq C_2$ if for any $t$ and $n$:
$$(p-1) \eta_t \left[ \left( 1 + \dfrac{C_2}{\eta_t } \right)^q - 1 - \dfrac{qC_2}{\eta_t} \right]+C_2 \geq \left(T-t+\left( \frac{\eta^\star}{n}\right)^{q-1} \right) b^\eta_t  + \left(T-t+\left( \frac{\eta^\star}{n}\right)^{q-1} \right)^{p} \gamma_t.$$
In particular it holds if 
\begin{align*}
& (p-1) \left[ \left( 1 + \dfrac{C_2}{\eta_t } \right)^q - 1 - \dfrac{qC_2}{\eta_t} \right]+\dfrac{C_2}{\eta_t}  =(p-1) \left[ \left( 1 + \dfrac{C_2}{\eta_t } \right)^q - 1 - \dfrac{C_2}{\eta_t} \right] \\
&\qquad  \geq \dfrac{1}{\eta_\star} (T+(\eta^\star)^{q-1}) \left[  \|b^\eta\|_\infty  + (T+(\eta^\star)^{q-1})^{p} \gamma^\star \right],
\end{align*}
which is equivalent to
$$
\left( 1 + \dfrac{C_2}{\eta_t } \right)^q - 1 - \dfrac{C_2}{\eta_t}  \geq \dfrac{q-1}{\eta_\star} (T+(\eta^\star)^{q-1}) \left[  \|b^\eta\|_\infty  + (T+(\eta^\star)^{q-1})^{p} \gamma^\star \right].$$
The function $\psi : x \mapsto (1+x)^q - 1 - x$ is continuous and increasing on $[0,+\infty[$, with $\psi(0)=0$ and $\psi(\infty) = \infty$. Hence
it is sufficient to take 
$$C_2 \geq \eta^\star \left[\psi^{-1}\left( \dfrac{q-1}{\eta_\star} (T+(\eta^\star)^{q-1}) \left[  \|b^\eta\|_\infty  + (T+(\eta^\star)^{q-1})^{p} \gamma^\star \right]\right)  \vee 1 \right],$$
to obtain, by the comparison result for BSDE, that a.s. for any $t$, $\mathcal H^n_t \leq V_t$. Remark that for $q= 2$, $C_2$ can be explicitly obtained by solving a second degree equation.

To get the lower estimate, we similarly proceed, by defining $U_t = -C_1 \left( T-t + \left(\frac{\eta^\star}{n}\right)^{q-1}\right)^2$. This quantity satisfies $-dU_t = -2 C_1 \left( T-t + \left(\frac{\eta^\star}{n}\right)^{q-1}\right) dt$ and 
    \begin{align} \nonumber
     F \left( t- \left(\dfrac{\eta^\star}{n}\right)^q , U_t  \right) 
     & = \left[ \left(T-t+\left( \frac{\eta^\star}{n}\right)^{q-1} \right) b^\eta_t  + \left(T-t+\left( \frac{\eta^\star}{n}\right)^{q-1} \right)^{p} \gamma_t \right] \\ \label{eq:comp_gene_control_H_n}
    & \qquad -(p-1) \eta_t \left[  \left( 1 - \chi_t \right)\left| 1 - \chi_t \right|^{q-1}  - 1 + q \chi_t \right] \geq -2 \eta_t \chi_t
       \end{align}
where 
$\chi_t = \dfrac{C_1}{\eta_t} \left( T-t + \left(\frac{\eta^\star}{n}\right)^{q-1}\right) \geq 0$. 
Note that for some $\delta > 0$ and $n_0 \in \mathbb N$ such that 
\begin{equation} \label{eq:first_cond_control_H_n}
C_1 \left( \delta + \left(\frac{\eta^\star}{n_0}\right)^{q-1}\right)\leq \dfrac{\eta_\star}{2},
\end{equation}
the quantity $\chi_t$ satisfies: $\chi_t \leq 1/2$. Moreover let us recall that if $1-\chi_t > 0$, then 
$$\left( 1 - \chi_t \right)\left| 1 - \chi_t \right|^{q-1}  - 1 + q \chi_t = q(q-1)\chi_t^2 \int_0^1 \left| 1 - a \chi_t \right|^{q-2} (1-a) da.$$
Then the previous inequality \eqref{eq:comp_gene_control_H_n} holds if 
$$ q\eta_t \chi_t^2 \int_0^1 \left| 1 - a \chi_t \right|^{q-2} (1-a) da  - 2 \eta_t \chi_t \leq  - \left(T-t+\left( \frac{\eta^\star}{n}\right)^{q-1} \right) \left[  \|b^\eta\|_\infty  + (T+(\eta^\star)^q)^{p-1} \gamma^\star  \right],$$
which can be written as:
$$ (C_1)^2 \dfrac{\zeta_t}{\eta_t} \left(T-t+\left( \frac{\eta^\star}{n}\right)^{q-1} \right)  - 2 C_1 + K \leq  0 ,$$
with 
$$ \zeta_t = q \int_0^1 \left| 1 - a \chi_t \right|^{q-2} (1-a) da, \quad K= \left[  \|b^\eta\|_\infty  + (T+(\eta^\star)^q)^{p-1} \gamma^\star  \right].$$
Note that for $q=2$, $\zeta_t = 1$, which simplifies the discussion here after. First under \eqref{eq:first_cond_control_H_n}, for any $T-\delta \leq t \leq T$ and $n \geq n_0$
$$\zeta_\star = \dfrac{q}{2} \left( \dfrac{1}{2^{q-2}} \wedge 1 \right) \leq \zeta_t= q \int_0^1 \left| 1 - a \chi_t \right|^{q-2} (1-a) da \leq \dfrac{q}{2} \left( \dfrac{1}{2^{q-2}} \vee 1 \right) = \zeta^\star.$$
Now if 
\begin{equation} \label{eq:second_cond_control_H_n}
 \delta + \left(\frac{\eta^\star}{n_0}\right)^{q-1} \leq \dfrac{\eta_\star}{K \zeta^\star},
\end{equation}
then for any $T-\delta \leq t \leq T$ and $n \geq n_0$, $\dfrac{\zeta_t}{\eta_t} K \left(T-t+\left( \frac{\eta^\star}{n}\right)^{q-1} \right)\leq 1$. 
Thus the desired estimate holds if 
\begin{equation} \label{eq:third_cond_control_H_n}
\dfrac{\eta_t}{\zeta_t\left(T-t+\left( \frac{\eta^\star}{n}\right)^{q-1} \right)} \left( 1 - \sqrt{1-\dfrac{\zeta_t}{\eta_t} K\left(T-t+\left( \frac{\eta^\star}{n}\right)^{q-1} \right)} \right) \leq C_1 
\end{equation}
and 
\begin{equation} \label{eq:fourth_cond_control_H_n}
C_1 \leq \dfrac{\eta_t}{\zeta_t\left(T-t+\left( \frac{\eta^\star}{n}\right)^{q-1} \right)} \left( 1 + \sqrt{1-\dfrac{\zeta_t}{\eta_t} K\left(T-t+\left( \frac{\eta^\star}{n}\right)^{q-1} \right)} \right).
\end{equation}
For the upper bound, remark that for $T-\delta \leq t \leq T$ and $n \geq n_0$
\begin{align*}
   & \dfrac{\eta_t}{\zeta_t\left(T-t+\left( \frac{\eta^\star}{n}\right)^{q-1} \right)} \left( 1 + \sqrt{1-\dfrac{\zeta_t}{\eta_t} K\left(T-t+\left( \frac{\eta^\star}{n}\right)^{q-1} \right)} \right) \\
   & \qquad \geq \dfrac{\eta_\star}{\zeta^\star \left(\delta+\left( \frac{\eta^\star}{n_0}\right)^{q-1} \right)} \left( 1 + \sqrt{1-\dfrac{\zeta^\star}{\eta_\star} K\left(\delta+\left( \frac{\eta^\star}{n_0}\right)^{q-1} \right)} \right). 
\end{align*}
And the lower bound satisfies:
\begin{align*}
    & \dfrac{\eta_t}{\zeta_t\left(T-t+\left( \frac{\eta^\star}{n}\right)^{q-1} \right)} \left( 1 - \sqrt{1-\dfrac{\zeta_t}{\eta_t} K\left(T-t+\left( \frac{\eta^\star}{n}\right)^{q-1} \right)} \right) \\
    & \quad =  \dfrac{K}{1 + \sqrt{1-\dfrac{\zeta_t}{\eta_t} K\left(T-t+\left( \frac{\eta^\star}{n}\right)^{q-1} \right)}} \leq \dfrac{K}{1 +\sqrt{1-\dfrac{\eta^\star}{\eta_\star} \left(\delta+\left( \frac{\eta^\star}{n_0}\right)^{q-1} \right)}}.
\end{align*}
Therefore there exists $\delta > 0$, $n_0\in \mathbb N$ and $C_1 > 0 $ such that Conditions \eqref{eq:first_cond_control_H_n},\eqref{eq:second_cond_control_H_n}, \eqref{eq:third_cond_control_H_n} and \eqref{eq:fourth_cond_control_H_n} hold. Then by comparison principle for BSDEs, for any $T-\delta \leq t \leq T$ and $n \geq n_0$
$$\mathcal H^n _t \geq -C_1 \left(T-t + \left( \frac{\eta^\star}{n} \right)^{q-1}\right)^2.$$

The last statement of the lemma comes from \eqref{eq:a_priori_estimate}: on $[0,T-\delta]$
\begin{align*}
0\leq Y^n_t \leq Y_t &  \leq \frac{1}{\delta^p} \left(  \frac{1}{\eta_\star} + (T+1)^{p} \gamma^\star \right) 
\end{align*}
and from the very definition \eqref{eq:definition:Hn} of $\mathcal H^n$: 
 \begin{equation*} 
|\mathcal H^n_t| \leq \left(T+ \left(\frac{\eta^\star}{n}\right)^{q-1}\right)^{p} Y^n_t + \eta^\star .
\end{equation*}
This achieves the proof of the lemma. 
\end{proof}

Since $Y^n_t$ converges a.s. to $Y_t$ for any $t$, we already know that $\mathcal H^n_t$ converges to $H_t$. The next result is a convergence result with weights. 
\begin{lemme} \label{lem:convergence_H_n}
We have a.s. 
$$\lim_{n\to +\infty} \int_0^T \left| \dfrac{H_t}{T-t} - \dfrac{\mathcal H^n_t}{\left(T-t+ \frac{\eta^\star}{n}\right)}  \right| dt = 0$$
and for any $\varrho > 0$
  $$\lim_{n\to +\infty} \mathbb E \int_0^T \left| \dfrac{H_t}{T-t} - \dfrac{\mathcal H^n_t}{\left(T-t+ \frac{\eta^\star}{n}\right)}  \right|^\varrho dt = 0.$$
\end{lemme}
\begin{proof}
For any $0 \leq t < T$ and any $n$, we have 
    \begin{eqnarray*}
        \dfrac{H_t}{T-t} - \dfrac{\mathcal H^n_t}{\left(T-t+ \left( \frac{\eta^\star}{n} \right)^{q-1}\right)} & = & 
        (T-t)^{p-1}Y_t - \left(T-t+\left( \frac{\eta^\star}{n} \right)^{q-1}\right)^{p-1} Y^n_t
    \end{eqnarray*}
    and $Y$ and $Y^n$ are non-negative processes. Using \eqref{eq:a_priori_estimate}, we have
\begin{align*}
0\leq (T-t)^{p-1}Y_t &  \leq \frac{1}{T-t}  \mathbb E^{\mathcal{F}_t} \left( \ \int_t^{T} \left( \frac{1}{\eta_s}  + (T-s+1)^{p} \gamma_s \right) ds \right) \leq \frac{1}{\eta_\star} + (T+1)^{p} \gamma^\star .
\end{align*}
    Now from the a priori estimate \eqref{eq:a_priori_estimate_Y_L}:
\begin{align*}
Y^n_t & \leq \frac{1}{(T-t+\frac{1}{n^{q-1}})^{p}} \left\{ \dfrac{1}{n^{q-1}} + \mathbb E^{\mathcal{F}_t} \left( \ \int_t^{T} \left( \left(  \frac{p-1}{\eta_s} \right)^{p-1} + (T-s+1)^{p} \gamma_s \right) ds \right)  \right\} \\
& \leq \frac{1}{(T-t+\frac{1}{n^{q-1}})^{p}} \left\{ \dfrac{1}{n^{q-1}} + (T-t)  \left( \left(  \frac{p-1}{\eta_\star} \right)^{p-1} + (T+1)^{p} \gamma^\star \right) \right\} . 
\end{align*}
Moreover 
\begin{align*}
\frac{\left(T-t+\left( \dfrac{\eta^\star}{n} \right)^{q-1}\right)^{p-1}}{n^{q-1}\left(T-t+\dfrac{1}{n^{q-1}}\right)^{p}} & =\left(\frac{T-t+\left( \dfrac{\eta^\star}{n} \right)^{q-1}}{T-t+\dfrac{1}{n^{q-1}}} \right)^{p-1} \frac{1}{n^{q-1}\left(T-t+\dfrac{1}{n^{q-1}}\right)} \\
& =\left( 1 + \frac{\left( \eta^\star \right)^{q-1} - 1}{n^{q-1}(T-t)+1} \right)^{p-1} \frac{1}{\left(n^{q-1}(T-t)+1\right)} \leq \eta^\star.
\end{align*}
And
\begin{align*}
\frac{\left(T-t+\left( \dfrac{\eta^\star}{n} \right)^{q-1}\right)^{p-1}}{\left(T-t+\dfrac{1}{n^{q-1}}\right)^{p}} (T-t)& =\left(\frac{T-t+\left( \dfrac{\eta^\star}{n} \right)^{q-1}}{T-t+\dfrac{1}{n^{q-1}}} \right)^{p-1} \frac{T-t}{\left(T-t+\dfrac{1}{n^{q-1}}\right)} \\
& =\left( 1 + \frac{\left( \eta^\star \right)^{q-1} - 1}{n^{q-1}(T-t)+1} \right)^{p-1} \frac{n^{q-1}(T-t)}{\left(n^{q-1}(T-t)+1\right)} \leq \eta^\star.
\end{align*}
Therefore 
$$\left(T-t+\left( \frac{\eta^\star}{n} \right)^{q-1}\right)^{p-1} Y^n_t\leq \eta^\star \left[ 1+ \left( \left(  \frac{p-1}{\eta_\star} \right)^{p-1} + (T+1)^{p} \gamma^\star \right) \right].$$
Since $Y^n$ converges to $Y$ pointwise, the result is a consequence of the dominated convergence theorem.
\end{proof}

Now let us study the Malliavin derivatives. We already know that $Y^n$ is Malliavin differentiable and the solution of the linear BSDE \eqref{eq:BSDE_Malliavin_deriv} is  for $0\leq \theta \leq t \leq T$:
$$D_\theta Y^n_t= \mathbb E^{\mathcal{F}_t} \left[\int_t^T \left(  \left( \dfrac{Y^n_s}{\eta_s} \right)^q D_\theta \eta_s + D_\theta \gamma_s \right)\exp\left( -p \int_t^s \left( \dfrac{Y^n_u}{\eta_u} \right)^{q-1}du\right)  ds \right].$$
\begin{proposition} \label{propo:DYn:Deta:DHn}
    The processes $(\mathcal{H}^n)_{n\in \mathbb{N}}$ are Malliavin differentiable and their Malliavin derivatives satisfy, for any $0\leq \theta \leq t \leq T$:
    $$D_\theta Y^n_t= \dfrac{1}{\left(T-t+ \left(\frac{\eta^\star}{n}\right)^{q-1}\right)^{p-1}} D_\theta \eta_t + \dfrac{1}{\left(T-t+ \left(\frac{\eta^\star}{n}\right)^{q-1}\right)^{p}}D_\theta \mathcal H^n_t.$$
\end{proposition}

\begin{proof}
    It is due to the fact that the processes $Y^n$ and $\eta$ are Malliavin differentiable and \eqref{eq:definition:Hn}.
\end{proof}

Evoke that $D_\theta H$ is the solution of the BSDE \eqref{eq:BSDE_Malliavin_deriv_H}:
    $$D_\theta H_t = \int_t^T  \left[ (T-s) D_\theta b^\eta_s  + (T-s)^{p} D_\theta \gamma_s - \dfrac{\partial G}{\partial \eta}(s,H_s,\eta_s) D_\theta \eta_s - \dfrac{\partial G}{\partial h}(s,H_s,\eta_s)  D_\theta H_s \right]ds - \int_t^T D_\theta Z^H_s dW_s.$$
The Malliavin derivative $D_\theta \mathcal H^n$ is the solution of:
    \begin{align*}
    D_\theta \mathcal H^n_t & =- \left(   \dfrac{\eta^\star}{n}\right)^{q-1}  D_\theta\eta_T + \int_t^T  \left[ \left(T-s+\left( \frac{\eta^\star}{n} \right)^{q-1}\right) D_\theta b^\eta_s  + \left(T-s+\left( \frac{\eta^\star}{n} \right)^{q-1}\right)^{p} D_\theta \gamma_s \right] ds \\
    & - \int_t^T  \left[\dfrac{\partial G}{\partial \eta}\left(s-\left( \frac{\eta^\star}{n} \right)^{q-1},\mathcal H^n_s,\eta_s\right) D_\theta \eta_s + \dfrac{\partial G}{\partial h}\left(s-\left( \frac{\eta^\star}{n} \right)^{q-1},\mathcal H^n_s,\eta_s\right)  D_\theta \mathcal H^n_s \right]ds - \int_t^T D_\theta \mathcal Z^n_s dW_s.
    \end{align*}
    
\subsection{Convergence of $D_\theta \mathcal H^n$}
  
Our goal is to prove that $D_\theta \mathcal H^n$ converges to $D_\theta H$. The difference $D_\theta H- D_\theta \mathcal H^n$ satisfies the next equation:
\begin{align} \nonumber
 \Delta^n_t =   D_\theta H_t - D_\theta \mathcal H^n_t & = \left(   \dfrac{\eta^\star}{n}\right)^{q-1}  D_\theta\eta_T -\left(   \dfrac{\eta^\star}{n}\right)^{q-1} \int_t^T  \left[ D_\theta b^\eta_s + p \left( \int_0^1 \left(T-s+ a\left(   \dfrac{\eta^\star}{n}\right)^{q-1} \right) da \right) D_\theta \gamma_s \right] ds \\  \nonumber
    & - \int_t^T \left[ \dfrac{\partial G}{\partial \eta}(s,H_s,\eta_s) - \dfrac{\partial G}{\partial \eta}\left(s-\left( \frac{\eta^\star}{n} \right)^{q-1},\mathcal H^n_s,\eta_s\right) \right] D_\theta \eta_s ds - \int_t^T (D_\theta Z^H_s - D_\theta \mathcal Z^n_s) dW_s\\  \nonumber 
    & -  \int_t^T \left[ \dfrac{\partial G}{\partial h}(s,H_s,\eta_s)  - \dfrac{\partial G}{\partial h}\left(s-\left( \frac{\eta^\star}{n} \right)^{q-1},\mathcal H^n_s,\eta_s\right) \right] D_\theta H_s ds \\ \label{eq:diff_mallia_BSDE}
    & -  \int_t^T  \dfrac{\partial G}{\partial h}\left(s-\left( \frac{\eta^\star}{n} \right)^{q-1},\mathcal H^n_s,\eta_s\right)  \left[ D_\theta H_s - D_\theta \mathcal H^n_s \right] ds.
\end{align}
Our first statement is:
\begin{lemme} \label{lem:convergence:DHn:control_1}
  There exists a constant $\kappa_1>0$ such that a.s. for any $n \geq n_0$ and $s \in [0,T]$
  $$ \dfrac{\partial G}{\partial h}\left(s-\left( \frac{\eta^\star}{n} \right)^{q-1},\mathcal H^n_s,\eta_s\right)  \geq - \kappa_1.$$
\end{lemme}
\begin{proof}
Evoke that from  \eqref{eq:deriv_G_h}
$$\dfrac{\partial G}{\partial h}(t,h,\eta) = \dfrac{qh}{\eta (T-t)^2} \int_0^1\left| 1 + a\dfrac{1}{\eta (T-t)} h\right|^{q-2} \mbox{sign } \left(1 + a\dfrac{1}{\eta (T-t)} h\right)  da.$$
Lemma \ref{lem:bounds_on_H^n} implies that for any $T-\delta \leq s \leq T$ and $n \geq n_0$:
\begin{equation}\label{eq:estim_diff_deriv_part_h}
-\dfrac{C_1 \zeta^\star}{\eta_\star}\leq -\dfrac{C_1 \zeta^\star}{\eta_s}\leq \dfrac{\partial G}{\partial h}\left(s-\left( \frac{\eta^\star}{n} \right)^{q-1},\mathcal H^n_s,\eta_s\right) \leq  \dfrac{C_2 \zeta^\star}{\eta_s \left(T-s+\left( \frac{\eta^\star}{n} \right)^{q-1}\right)}.
\end{equation}
Now from \eqref{eq:deriv_G_h_0} and since $\mathcal H^n$ is bounded by $C$, if $s \leq T-\delta$
 \begin{align*}
\left|   \dfrac{\partial G}{\partial h}\left(s-\left( \frac{\eta^\star}{n} \right)^{q-1},\mathcal H^n_s,\eta_s\right) \right| & = \dfrac{p}{ (T-s+\left( \frac{\eta^\star}{n} \right)^{q-1})} \left| \left| 1 + \dfrac{1}{\eta_s (T-s+\left( \frac{\eta^\star}{n} \right)^{q-1})} \mathcal H^n_s\right|^{q-1} -  1 \right| \\
& \leq \dfrac{p}{\delta}  \left( \left( 1 + \dfrac{C}{\eta_\star (T-\delta} \right)^{q-1} +  1 \right). 
\end{align*}
\end{proof}

The second result concerns the control of the next difference terms. 
\begin{lemme}  \label{lem:convergence:DHn:control_2}
There exist two processes $\Upsilon^n$ and $\widetilde \Upsilon^n$ such that 
\begin{equation}  \label{eq:estim_diff_deriv_eta_0}
\left| \dfrac{\partial G}{\partial \eta}(s,H_s,\eta_s) - \dfrac{\partial G}{\partial \eta}\left(s-\left( \frac{\eta^\star}{n} \right)^{q-1},\mathcal H^n_s,\eta_s\right)  \right| \leq \Upsilon^n_s \left| \dfrac{H_s}{(T-s)} - \dfrac{\mathcal H^n_s}{(T-s+(\eta^\star/n)^{q-1})}\right|^{(q-1)\wedge 1}
\end{equation}
and
\begin{align} \nonumber 
& \left|  \dfrac{\partial G}{\partial h}(s,H_s,\eta_s) - \dfrac{\partial G}{\partial h}\left(s-\left( \frac{\eta^\star}{n} \right)^{q-1},\mathcal H^n_s,\eta_s\right)  \right|\\ \label{eq:estim_diff_deriv_h_0}
& \quad  \leq \dfrac{\widetilde \Upsilon^n_s}{(T-s)} \left[ \left| \dfrac{H_s}{(T-s)} - \dfrac{\mathcal H^n_s}{(T-s+(\eta^\star/n)^{q-1})}\right|^{(q-1) \wedge 1} +\left( \frac{\eta^\star}{n} \right)^{q-1} \dfrac{p}{(T-s+(\eta^\star/n)^{q-1})}  \right]. 
\end{align}
And $\Upsilon^n$ and $\widetilde \Upsilon^n$  are bounded uniformly w.r.t $s\in [0,T]$ and $n\geq n_0$. 
\end{lemme}
\begin{proof}
For $(s,h)$ and $(u,\hat h)$ and if 
$$\delta = \dfrac{h}{\eta (T-s)}- \dfrac{\hat h}{\eta (T-u)}$$
then from \eqref{eq:deriv_G_eta_0} and \eqref{eq:deriv_G_eta}
\begin{align} \nonumber 
&  \dfrac{\partial G}{\partial \eta}(s,h,\eta) - \dfrac{\partial G}{\partial \eta}\left(u,\hat h,\eta\right)  \\ \label{eq:repres_diff_deriv_eta}
& \quad = -q \delta \int_0^1 \left| 1 + \dfrac{\hat h}{\eta (T-u)} + a \delta \right|^{q-1}  da + p  \left| 1 + \dfrac{h}{\eta (T-s)} \right|^{q-1} -p  \left| 1 + \dfrac{\hat h}{\eta (T-u)} \right|^{q-1}  \\  \nonumber
& \quad =  -q \delta \int_0^1 \left| 1 + \dfrac{\hat h}{\eta (T-u)} + a \delta \right|^{q-1}  da  + p (q-1) \delta \int_0^1 \left| 1 + \dfrac{\hat h}{\eta (T-u)} + a \delta \right|^{q-2}  \mbox{sign }  \left( 1 + \dfrac{\hat h}{\eta (T-u)} + a \delta\right) da \\
& \quad = -q \delta \int_0^1 \left| 1 + \dfrac{\hat h}{\eta (T-u)} + a \delta \right|^{q-2} \left( \dfrac{\hat h}{\eta (T-u)} + a \delta \right)  \mbox{sign }  \left( 1 + \dfrac{\hat h}{\eta (T-u)} + a \delta\right)   da
\end{align}
provided that there is no division by zero when $q < 2$. 

If $(s,h)$ and $(u,\hat h)$ are chosen such that 
$$\dfrac{h}{\eta (T-s)} > -\dfrac{1}{2}, \quad \dfrac{\hat h}{\eta (T-u)} > -\dfrac{1}{2},$$
for any $a \in [0,1]$
$$1 + \dfrac{\hat h}{\eta (T-u)} + a \delta  = 1+ (1-a)\dfrac{\hat h}{\eta (T-u)} + a \dfrac{h}{\eta (T-s)} > \dfrac{1}{2}.$$
Evoke that  from our choice of $\delta$ and $n_0$, 
$$\dfrac{H_s}{\eta_s (T-s)} > -\dfrac{1}{2}, \quad \dfrac{\mathcal H^n_s}{\eta_s(T-s+(\eta^\star/n)^{q-1})} > -\dfrac{1}{2},$$
and from Lemma \ref{lem:bounds_on_H^n}
$$\dfrac{1}{2} \leq  1 +\dfrac{\mathcal H^n_s}{\eta_s(T-s+(\eta^\star/n)^{q-1})} + a \delta^n_s  \leq 1 + \dfrac{C_2+R}{\eta_\star},$$
with 
$$\delta^n_s = \dfrac{H_s}{\eta_s (T-s)} - \dfrac{\mathcal H^n_s}{\eta_s(T-s+(\eta^\star/n)^{q-1})} . $$
Hence we deduce that for any $T-\delta \leq s \leq T$ and $n \geq n_0$
$$ \dfrac{\partial G}{\partial \eta}(s,H_s,\eta_s) - \dfrac{\partial G}{\partial \eta}\left(s-\left( \frac{\eta^\star}{n} \right)^{q-1},\mathcal H^n_s,\eta_s\right)  =  \Upsilon^n_s  \delta^n_s $$
with the bounded process
$$ \Upsilon^n_s = -q \int_0^1 \left( 1 +\dfrac{\mathcal H^n_s}{\eta_s(T-s+(\eta^\star/n)^{q-1})} + a\delta^n_s  \right)^{q-2} \left( \dfrac{\mathcal H^n_s}{\eta_s(T-s+(\eta^\star/n)^{q-1})}+ a\delta^n_s \right)  da .$$
On the rest of the time interval $[0,T-\delta]$, since $s \mapsto 1/(T-s)$, $H$ and $\mathcal H^n$ are bounded, from \eqref{eq:repres_diff_deriv_eta}, we easily deduce that if $q \geq 2$, the functions are Lipschitz continuous and thus
$$\left| \dfrac{\partial G}{\partial \eta}(s,H_s,\eta_s) - \dfrac{\partial G}{\partial \eta}\left(s-\left( \frac{\eta^\star}{n} \right)^{q-1},\mathcal H^n_s,\eta_s\right) \right| \leq \Upsilon^n_s |\delta^n_s|.$$
But for $1 < q < 2$, since $x\mapsto |x|^{q-1}$ is $(q-1)$-H\"older continuous, we only obtain that 
$$\left| \dfrac{\partial G}{\partial \eta}(s,H_s,\eta_s) - \dfrac{\partial G}{\partial \eta}\left(s-\left( \frac{\eta^\star}{n} \right)^{q-1},\mathcal H^n_s,\eta_s\right) \right| \leq \Upsilon^n_s ( |\delta^n_s|^{q-1} \vee |\delta^n_s|) .$$

Similarly with \eqref{eq:deriv_G_h_0} and the same previous notations as previously:
\begin{align*}
&  \dfrac{\partial G}{\partial h}(s,h,\eta) - \dfrac{\partial G}{\partial h}\left(u,\hat h,\eta\right)  = \dfrac{p}{ (T-s)} \left| 1 + \dfrac{h}{\eta (T-s)} \right|^{q-1} -\dfrac{p}{ (T-u)} \left| 1 + \dfrac{\hat h}{\eta (T-u)} \right|^{q-1} \\
& \quad =  \dfrac{p}{ (T-s)} \left| 1 + \dfrac{h}{\eta (T-s)} \right|^{q-1} -\dfrac{p}{ (T-s)} \left| 1 + \dfrac{\hat h}{\eta (T-u)} \right|^{q-1} + \left(  \dfrac{p}{ (T-s)} -  \dfrac{p}{ (T-u)}\right)  \left| 1 + \dfrac{\hat h}{\eta (T-u)} \right|^{q-1}  \\
& \quad =  \dfrac{p (q-1)}{ (T-s)}  \delta \int_0^1 \left| 1 + \dfrac{\hat h}{\eta (T-u)} + a \delta \right|^{q-2} \mbox{sign }\left( 1 + \dfrac{\hat h}{\eta (T-u)} + a \delta\right)  da  \dfrac{p(s-u)}{ (T-s)(T-u)} \left| 1 + \dfrac{\hat h}{\eta (T-u)} \right|^{q-1} .
\end{align*}
Hence there exists a uniformly bounded process $\widetilde \Upsilon^n$ such that a.s. on $[T-\delta,T]$ and for $n \geq n_0$
\begin{align*}
&  \dfrac{\partial G}{\partial h}(s,H_s,\eta_s) - \dfrac{\partial G}{\partial h}\left(s-\left( \frac{\eta^\star}{n} \right)^{q-1},\mathcal H^n_s,\eta_s\right)  \\
& \quad  = \dfrac{\widetilde \Upsilon^n_s}{(T-s)} \left[ \left( \dfrac{H_s}{(T-s)} - \dfrac{\mathcal H^n_s}{(T-s+(\eta^\star/n)^{q-1})}\right) +\left( \frac{\eta^\star}{n} \right)^{q-1} \dfrac{p}{(T-s+(\eta^\star/n)^{q-1})}  \right]. 
\end{align*}
And on $[0,T-\delta]$, we obtain the inequality by Lipschitz-continuity if $q \geq 2$ or H\"older-continuity if $q < 2$, and by boundedness of $H$ and $\mathcal H^n$. 
\end{proof}

Let us start with the next result:
\begin{lemme} \label{lem:convergence:DHn}
    Assume that for some $\varrho > 1$,
    $$\mathbb E \int_0^T  \left[ |D_\theta \eta_T|^\varrho +  \int_0^T \left( |D_\theta b^\eta_s|^\varrho + |D_\theta \gamma_s|^\varrho + |D_\theta \eta_s|^\varrho \right) ds\right] d\theta < +\infty. $$
    Then for any $0\leq \theta \leq t \leq T$ $D_\theta \mathcal H^n_t$  converges a.s. to $D_\theta H_t$ and for $\ell < \varrho$, $D \mathcal H^n$ converges to $D H$ in $L^\ell(\Omega \times [0,T]^2)$. 
\end{lemme}
\begin{proof}
From Lemma \ref{lem:convergence:DHn:control_1}
$$\Gamma^n_{t,s} = \exp \left[ - \int_t^s  \dfrac{\partial G}{\partial h}\left(u-\left( \frac{\eta^\star}{n} \right)^{q-1},\mathcal H^n_u,\eta_u\right) du \right],$$
is bounded uniformly w.r.t. $n\geq n_0$ by $\exp( T \kappa_1)$. 
 Since \eqref{eq:diff_mallia_BSDE} is a linear BSDE, we have:
\begin{align} \nonumber
    D_\theta H_t - D_\theta \mathcal H^n_t & =\left(   \dfrac{\eta^\star}{n}\right)^{q-1} \mathbb E^{\mathcal{F}_t} \left[  D_\theta\eta_T \Gamma^n_{t,T} \right] \\ \nonumber
    & -\left(   \dfrac{\eta^\star}{n}\right)^{q-1}   \mathbb E^{\mathcal{F}_t} \left[   \int_t^T  \left[ D_\theta b^\eta_s + p \left( \int_0^1 \left(T-s+ a\left(   \dfrac{\eta^\star}{n}\right)^{q-1} \right) da \right) D_\theta \gamma_s \right] \Gamma_{t,s}^n ds \right] \\  \nonumber
    & - \mathbb E^{\mathcal{F}_t} \left[  \int_t^T \left[ \dfrac{\partial G}{\partial \eta}(s,H_s,\eta_s) - \dfrac{\partial G}{\partial \eta}\left(s-\left( \frac{\eta^\star}{n} \right)^{q-1},\mathcal H^n_s,\eta_s\right) \right] D_\theta \eta_s  \Gamma_{t,s}^n ds \right]\\   \label{eq:diff_mallia_BSDE_2}
    & -  \mathbb E^{\mathcal{F}_t} \left[  \int_t^T \left[ \dfrac{\partial G}{\partial h}(s,H_s,\eta_s)  - \dfrac{\partial G}{\partial h}\left(s-\left( \frac{\eta^\star}{n} \right)^{q-1},\mathcal H^n_s,\eta_s\right) \right] D_\theta H_s  \Gamma_{t,s}^n ds \right] .
\end{align}
In the right-hand side of \eqref{eq:diff_mallia_BSDE_2}, the first two terms satisfy for any $0\leq \theta \leq t \leq T$ and $n \geq n_0$
\begin{equation}  \label{eq:estim_term_cond_1}
\left(   \dfrac{\eta^\star}{n}\right)^{q-1}  \left| \mathbb E^{\mathcal{F}_t} \left[  D_\theta\eta_T \Gamma^n_{t,T} \right] \right| \leq C \left(   \dfrac{\eta^\star}{n}\right)^{q-1}  \mathbb E^{\mathcal{F}_t} \left[ |D_\theta\eta_T| \right]
\end{equation}
and 
\begin{align}\nonumber
&\left(   \dfrac{\eta^\star}{n}\right)^{q-1}  \left| \mathbb E^{\mathcal{F}_t} \left[   \int_t^T  \left[ D_\theta b^\eta_s + p \left( \int_0^1 \left(T-s+ a\left(   \dfrac{\eta^\star}{n}\right)^{q-1} \right) da \right) D_\theta \gamma_s \right] \Gamma_{t,s}^n ds \right] \right| \\ \nonumber 
& \quad \leq  C \left( \dfrac{\eta^\star}{n}\right)^{q-1} \mathbb E^{\mathcal{F}_t} \left[ \int_t^T  \left( |D_\theta b^\eta_s| + p \left(T+ (\eta^\star)^{q-1} \right) |D_\theta \gamma_s| \right) ds \right] \\  \label{eq:estim_term_cond_2}
& \quad \leq  C \left(   \dfrac{\eta^\star}{n}\right)^{q-1} \mathbb E^{\mathcal{F}_t} \left[ \int_t^T  \left( |D_\theta b^\eta_s| +  |D_\theta \gamma_s| \right) ds  \right].
\end{align}
From \eqref{eq:estim_diff_deriv_eta_0}, we obtain that 
\begin{align} \nonumber
& \left|  \mathbb E^{\mathcal{F}_t} \left[  \int_t^T \left[ \dfrac{\partial G}{\partial \eta}(s,H_s,\eta_s) - \dfrac{\partial G}{\partial \eta}\left(s-\left( \frac{\eta^\star}{n} \right)^{q-1},\mathcal H^n_s,\eta_s\right) \right] D_\theta \eta_s  \Gamma_{t,s}^n ds \right] \right| \\ \nonumber
& \quad \leq \mathbb E^{\mathcal{F}_t}\left[  \int_t^T \Upsilon^n_s \left| \dfrac{H_s}{(T-s)} - \dfrac{\mathcal H^n_s}{(T-s+(\eta^\star/n)^{q-1})}\right|^{(q-1)\wedge 1} |D_\theta \eta_s|  \Gamma_{t,s}^n ds \right] \\ \label{eq:estim_diff_deriv_part_eta}
& \quad \leq C \mathbb E^{\mathcal{F}_t} \left[  \int_t^T  \left| \dfrac{H_s}{(T-s)} - \dfrac{\mathcal H^n_s}{(T-s+(\eta^\star/n)^{q-1})}\right|^{(q-1)\wedge 1} |D_\theta \eta_s| ds \right] .
\end{align}
From \eqref{eq:estim_Mallia_deriv_H}, we have
$|D_\theta H_s | \leq C (T-s) \zeta_s,$
where 
$$\zeta_s=  \mathbb E^{\mathcal{F}_s} \left[\int_s^T  \left( | D_\theta b^\eta_u|  +  |D_\theta \gamma_u| + | D_\theta \eta_u | \right)  du \right] = \mathbb E^{\mathcal{F}_s} \left[\int_s^T  \varpi_u du \right].$$
Thus using \eqref{eq:estim_diff_deriv_h_0}
\begin{align*}\nonumber
& \left| \dfrac{\partial G}{\partial h}(s,H_s,\eta_s) - \dfrac{\partial G}{\partial h}\left(s-\left( \frac{\eta^\star}{n} \right)^{q-1},\mathcal H^n_s,\eta_s\right) \right| |D_\theta H_s | \\ 
& \quad  \leq C \zeta_s |\widetilde \Upsilon^n_s| \left| \dfrac{H_s}{(T-s)} - \dfrac{\mathcal H^n_s}{(T-s+(\eta^\star/n)^{q-1})}\right|^{(q-1)\wedge 1}+ C p |\widetilde \Upsilon^n_s|  \left( \frac{\eta^\star}{n} \right)^{q-1} \dfrac{\zeta_s}{(T-s+(\eta^\star/n)^{q-1})}  
\end{align*}
and
\begin{align}\nonumber
& \left| \mathbb E^{\mathcal{F}_t} \left[  \int_t^T \left[ \dfrac{\partial G}{\partial h}(s,H_s,\eta_s)  - \dfrac{\partial G}{\partial h}\left(s-\left( \frac{\eta^\star}{n} \right)^{q-1},\mathcal H^n_s,\eta_s\right) \right] D_\theta H_s  \Gamma_{t,s}^n ds \right] \right| \\ \nonumber
&\quad  \leq  C \mathbb E^{\mathcal{F}_t} \left[  \int_t^T  \zeta_s  \left| \dfrac{H_s}{(T-s)} - \dfrac{\mathcal H^n_s}{(T-s+(\eta^\star/n)^{q-1})}\right|^{(q-1)\wedge 1} ds \right]  \\ \label{eq:estim_diff_deriv_part_h_2}
& \qquad +p C \left( \frac{\eta^\star}{n} \right)^{q-1}  \mathbb E^{\mathcal{F}_t} \left[  \int_t^T   \dfrac{\zeta_s}{(T-s+(\eta^\star/n)^{q-1})} ds \right] . 
\end{align}
For the second one 
\begin{align*}
    0 & \leq \mathbb E^{\mathcal{F}_t} \left[\int_t^T  \frac{\zeta_s}{T-s+ (\eta^\star/n)^{q-1}}  ds  \right]  =  \mathbb E^{\mathcal{F}_t} \left[\int_t^T  \frac{1}{T-s+ (\eta^\star/n)^{q-1}} \left( \int_s^T  \varpi_u du \right) ds \right] \\
    & = \ln(T-t+(\eta^\star/n)^{q-1}) \zeta_t  - \mathbb E^{\mathcal{F}_t} \left[\int_t^T \ln(T-s+(\eta^\star/n)^{q-1}) \varpi_s ds \right]\\
    & \leq  \ln(T+(\eta^\star)^{q-1}) \zeta_t  + \mathbb E^{\mathcal{F}_t} \left[\int_t^T |\ln(T-s)| \varpi_s ds \right] .
\end{align*}
Coming back to \eqref{eq:diff_mallia_BSDE_2}, with \eqref{eq:estim_term_cond_1}, \eqref{eq:estim_term_cond_2}, \eqref{eq:estim_diff_deriv_part_eta}, \eqref{eq:estim_diff_deriv_part_h_2}, we obtain that for $0\leq \theta \leq t \leq T$
\begin{align} \nonumber
   & \left| D_\theta H_t - D_\theta \mathcal H^n_t  \right| \leq C \left(   \dfrac{\eta^\star}{n}\right)^{q-1} \mathbb E^{\mathcal{F}_t} \left[ |D_\theta\eta_T| + 2 \int_t^T  \left( |D_\theta b^\eta_s| + |D_\theta \gamma_s| + | D_\theta \eta_s|\right) ds  \right] \\ \nonumber
    &+ C \left( \frac{\eta^\star}{n} \right)^{q-1}  \mathbb E^{\mathcal{F}_t} \left[\int_t^T |\ln(T-s)| \varpi_s ds \right]  + C \mathbb E^{\mathcal{F}_t} \left[  \int_t^T  \left| \dfrac{H_s}{(T-s)} - \dfrac{\mathcal H^n_s}{(T-s+(\eta^\star/n)^{q-1})}\right|^{(q-1)\wedge 1} |D_\theta \eta_s|  ds \right] \\   \label{eq:diff_DH_1}
    & + C \mathbb E^{\mathcal{F}_t} \left[  \int_t^T  \zeta_s  \left| \dfrac{H_s}{(T-s)} - \dfrac{\mathcal H^n_s}{(T-s+(\eta^\star/n)^{q-1})}\right|^{(q-1)\wedge 1} ds \right] .
\end{align}
With H\"older's inequality ($1/\varrho + 1 /\varrho_* = 1$):
\begin{align*}
 \mathbb E^{\mathcal{F}_t} \left[\int_t^T |\ln(T-s)| \varpi_s ds \right] & \leq \left[\int_t^T |\ln(T-s)|^{\varrho_*} ds  \right]^{1/\varrho_*} \left[ \mathbb E^{\mathcal{F}_t} \int_t^T  (\varpi_s)^\varrho  ds \right]^{1/\varrho}  \leq C  \left[ \mathbb E^{\mathcal{F}_t} \int_t^T  (\varpi_s)^\varrho  ds \right]^{1/\varrho}
\end{align*}
and
\begin{align*}
&  \mathbb E^{\mathcal{F}_t} \left[  \int_t^T  \left| \dfrac{H_s}{(T-s)} - \dfrac{\mathcal H^n_s}{(T-s+(\eta^\star/n)^{q-1})}\right|^{(q-1)\wedge 1} |D_\theta \eta_s|  ds \right] \leq \nu^n_t  \left[  \mathbb E^{\mathcal{F}_t}  \int_t^T  |D_\theta \eta_s|^\varrho  ds \right]^{1/\varrho} \\
&  \mathbb E^{\mathcal{F}_t} \left[  \int_t^T  \zeta_s  \left| \dfrac{H_s}{(T-s)} - \dfrac{\mathcal H^n_s}{(T-s+(\eta^\star/n)^{q-1})}\right|^{(q-1)\wedge 1} ds \right] \leq \nu^n_t  \left[  \mathbb E^{\mathcal{F}_t} \int_t^T  |\zeta_s |^\varrho ds \right]^{1/\varrho} 
 \end{align*}
 with 
 $$\nu^n_t = \left[ \mathbb E^{\mathcal{F}_t} \int_t^T  \left| \dfrac{H_s}{(T-s)} - \dfrac{\mathcal H^n_s}{(T-s+(\eta^\star/n)^{q-1})}\right|^{\varrho^* ((q-1)\wedge 1)}  ds\right]^{1/\varrho^*}  .$$
Note that 
 \begin{align*}
 \mathbb E^{\mathcal{F}_t} \left[  \int_t^T |\zeta_s |^\varrho  ds \right] & =\mathbb E^{\mathcal{F}_t} \left[  \int_t^T \mathbb E^{\mathcal{F}_s} \left[\int_s^T  \varpi_u du \right]^\varrho  ds  \right]  \\
 & \leq \mathbb E^{\mathcal{F}_t} \left[  \int_t^T (T-s)^{\varrho-1} \mathbb E^{\mathcal{F}_s} \left[\int_s^T ( \varpi_u)^\varrho  du \right]  ds \right]  \leq \dfrac{1}{\varrho} (T-t)^{\varrho} \mathbb E^{\mathcal{F}_t} \left[  \int_t^T ( \varpi_u)^\varrho  du \right] .
 \end{align*} 
From Lemma \ref{lem:convergence_H_n} and its proof, we can use the dominated convergence theorem to deduce that for a fixed $t$, $\nu^n$ converges to zero a.s. Thus coming back to \eqref{eq:diff_DH_1}, we obtain the a.s. convergence of $D_\theta \mathcal H^n_t$ to $D_\theta H_t$. 
 
 \medskip 
To obtain the convergence in mean, we raise to the power $1 < \ell < \varrho$ and use the expectation in \eqref{eq:diff_DH_1}:
\begin{align*} 
  & \mathbb E \left| D_\theta H_t - D_\theta \mathcal H^n_t  \right|^\ell \leq C \left(   \dfrac{\eta^\star}{n}\right)^{\ell (q-1)} \mathbb E \left[  |D_\theta\eta_T|^\ell +  \int_t^T  \left( |D_\theta b^\eta_s|^\ell  +  |D_\theta \gamma_s|^\ell  + | D_\theta \eta_s|^\ell  \right) ds \right] \\
  & + C \left( \frac{\eta^\star}{n} \right)^{\ell (q-1)}  \left[ \mathbb E \int_t^T  (\varpi_s)^\varrho  ds  \right]^{\ell/\varrho} + C \mathbb E \left[  (\nu^n_t)^\ell  \left[  \mathbb E^{\mathcal{F}_t}  \int_t^T  |D_\theta \eta_s|^\varrho  ds  \right]^{\ell/\varrho} \right]  + C \mathbb E \left[ (\nu^n_t)^\ell  \mathbb E^{\mathcal{F}_t} \left[  \int_t^T ( \varpi_u)^\varrho  du \right]^{\ell/\rho}\right] \\
    &  \leq C \left(   \dfrac{\eta^\star}{n}\right)^{\ell (q-1)} \mathbb E \left[  |D_\theta\eta_T|^\varrho +  \int_t^T  \left( |D_\theta b^\eta_s|^\varrho  +  |D_\theta \gamma_s|^\varrho + | D_\theta \eta_s|^\varrho \right) ds \right] \\
    & + C  \left[ \mathbb E (\nu^n_t)^{\ell \varrho / (\varrho-\ell)} \right]^{(\varrho-\ell)/\varrho} \left\{  \left[  \mathbb E  \int_t^T  |D_\theta \eta_s|^\varrho  ds \right]^{\ell/\varrho}   + \left[  \mathbb E \int_t^T ( \varpi_u)^\varrho  du\right]^{\ell/\rho}\right\} \\
    & \leq C\varepsilon_n  \mathbb E \left[  |D_\theta\eta_T|^\varrho +  \int_t^T  \left( |D_\theta b^\eta_s|^\varrho  +  |D_\theta \gamma_s|^\varrho + | D_\theta \eta_s|^\varrho \right) ds \right]
\end{align*}
 with 
 $$\varepsilon_n =  \left(   \dfrac{\eta^\star}{n}\right)^{\ell (q-1)} +  \left[ \mathbb E (\nu^n_0)^{\ell \varrho / (\varrho-\ell)} \right]^{(\varrho-\ell)/\varrho}.$$
 Therefore with Lemma \ref{lem:convergence_H_n}
 $$\lim_{n\to +\infty} \mathbb E \int_0^T \int_0^T  \left| D_\theta H_t - D_\theta \mathcal H^n_t  \right|^\ell d\theta dt   = 0,$$
 which achieves the proof of the lemma. 
\end{proof}

Remark that if the condition is: 
 $$\sup_{\theta \in[0,T]} \mathbb E  \left[ |D_\theta \eta_T|^\varrho +  \int_0^T \left( |D_\theta b^\eta_s|^\varrho + |D_\theta \gamma_s|^\varrho + |D_\theta \eta_s|^\varrho \right) ds\right] < +\infty,$$
 then 
  $$\lim_{n\to +\infty} \sup_{\theta \in [0,T]}  \mathbb E  \int_0^T  \left| D_\theta H_t - D_\theta \mathcal H^n_t  \right|^\ell dt = 0.$$

Now we prove a stronger convergence result. 
\begin{lemme} \label{lem:unif_convergence_DHn}
    Assume that for some $\varrho > 1$,
    $$\sup_{\theta \in [0,T]} \mathbb E \left[ |D_\theta \eta_T|^{\varrho} +  \int_0^T \left( |D_\theta b^\eta_s|^{\varrho} + |D_\theta \gamma_s|^{\varrho} + |D_\theta \eta_s|^{\varrho} \right) ds\right] < +\infty. $$
        Then for any $1 < \ell < \varrho$
    $$\lim_{n\to +\infty} \sup_{\theta \in [0,T]} \mathbb E\left( \sup_{t\in [0,T]}  |D_\theta H_t- D_\theta \mathcal H^n_t |^\ell  \right) = 0.$$
\end{lemme}
\begin{proof}
    We apply It\^o's formula to $\Delta^n_t = D_\theta H_t - D_\theta \mathcal H^n_t$
with the function $x\mapsto |x|^\ell$ for $\ell> 1$, $0\leq \theta \leq t \leq T$ and $n \geq n_0$. Using the BSDE representation \eqref{eq:diff_mallia_BSDE}, we obtain 
    \begin{align*} \nonumber
    & e^{\mu t}|\Delta^n_t|^\ell +\dfrac{\ell(1\wedge (\ell-1))}{2} \int_t^T e^{\mu s}|\Delta^n_s|^{\ell-2} \mathbf 1_{\Delta^n_s\neq 0}(D_\theta Z^H_s - D_\theta \mathcal Z^n_s)^2 ds \\
    & \leq  \left(   \dfrac{\eta^\star}{n}\right)^{\ell(q-1)}  e^{\mu T}|D_\theta\eta_T|^\ell - \int_t^T \mu e^{\mu s}|\Delta^n_s|^\ell ds \\
    & -\ell  \left(   \dfrac{\eta^\star}{n}\right)^{q-1} \int_t^T e^{\mu s}|\Delta^n_s|^{\ell-2} \mathbf 1_{\Delta^n_s\neq 0} \left[ D_\theta b^\eta_s + p \left( \int_0^1 \left(T-s+ a\left(   \dfrac{\eta^\star}{n}\right)^{q-1} \right) da \right) D_\theta \gamma_s \right]  \Delta^n_s ds \\  \nonumber
    & + \ell \int_t^T e^{\mu s}|\Delta^n_s|^{\ell-2} \mathbf 1_{\Delta^n_s\neq 0}  \left[ \dfrac{\partial G}{\partial \eta}(s,H_s,\eta_s) - \dfrac{\partial G}{\partial \eta}\left(s-\left( \frac{\eta^\star}{n} \right)^{q-1},\mathcal H^n_s,\eta_s\right) \right] D_\theta \eta_s \Delta^n_s  ds \\
    & - \ell \int_t^T e^{\mu s}|\Delta^n_s|^{\ell-2} \mathbf 1_{\Delta^n_s\neq 0} \Delta^n_s(D_\theta Z^H_s - D_\theta \mathcal Z^n_s) dW_s\\  \nonumber 
    & - \ell \int_t^T e^{\mu s}|\Delta^n_s|^{\ell-2} \mathbf 1_{\Delta^n_s\neq 0}\ \left[ \dfrac{\partial G}{\partial h}(s,H_s,\eta_s)  - \dfrac{\partial G}{\partial h}\left(s-\left( \frac{\eta^\star}{n} \right)^{q-1},\mathcal H^n_s,\eta_s\right) \right]  D_\theta H_s \Delta^n_s ds \\
    & - \ell \int_t^T e^{\mu s}|\Delta^n_s|^{\ell-2} \mathbf 1_{\Delta^n_s\neq 0}  \dfrac{\partial G}{\partial h}\left(s-\left( \frac{\eta^\star}{n} \right)^{q-1},\mathcal H^n_s,\eta_s\right) \left( \Delta^n_s \right)^2 ds.
\end{align*}
With Young's inequality we have 
\begin{align*}
& \left(   \dfrac{\eta^\star}{n}\right)^{q-1} |\Delta^n_s|^{\ell-2} \mathbf 1_{\Delta^n_s\neq 0}\left| D_\theta b^\eta_s + p \left( \int_0^1 \left(T-s+ a\left(   \dfrac{\eta^\star}{n}\right)^{q-1} \right) da \right) D_\theta \gamma_s \right| |  \Delta^n_s | \\
& \qquad \leq \dfrac{1}{\ell} \left(   \dfrac{\eta^\star}{n}\right)^{\ell(q-1)}\left| D_\theta b^\eta_s + p  \left(T+ \left(   \dfrac{\eta^\star}{n}\right)^{q-1} \right)  D_\theta \gamma_s \right|^\ell + \dfrac{\ell-1}{\ell} |  \Delta^n_s |^{\ell} ,
\end{align*}
and
\begin{align*}
& |\Delta^n_s|^{\ell-2} \mathbf 1_{\Delta^n_s\neq 0}\left|   \dfrac{\partial G}{\partial h}(s,H_s,\eta_s)  - \dfrac{\partial G}{\partial h}\left(s-\left( \frac{\eta^\star}{n} \right)^{q-1},\mathcal H^n_s,\eta_s\right) \right| |D_\theta H_s|| \Delta^n_s | \\
& \qquad \leq  \dfrac{1}{\ell} \left| \dfrac{\partial G}{\partial h}(s,H_s,\eta_s)  - \dfrac{\partial G}{\partial h}\left(s-\left( \frac{\eta^\star}{n} \right)^{q-1},\mathcal H^n_s,\eta_s\right)  \right|^\ell |D_\theta H_s|^\ell + \dfrac{\ell-1}{\ell} |  \Delta^n_s |^\ell ,
\end{align*}
and finally
\begin{align*}
 & |\Delta^n_s|^{\ell-2} \mathbf 1_{\Delta^n_s\neq 0} \left| \dfrac{\partial G}{\partial \eta}(s,H_s,\eta_s) - \dfrac{\partial G}{\partial \eta}\left(s-\left( \frac{\eta^\star}{n} \right)^{q-1},\mathcal H^n_s,\eta_s\right) \right| | D_\theta \eta_s|| \Delta^n_s | \\
 & \quad \leq \dfrac{1}{\ell} \left|\dfrac{\partial G}{\partial \eta}(s,H_s,\eta_s) - \dfrac{\partial G}{\partial \eta}\left(s-\left( \frac{\eta^\star}{n} \right)^{q-1},\mathcal H^n_s,\eta_s\right) \right|^\ell | D_\theta \eta_s|^\ell + \dfrac{\ell-1}{\ell}| \Delta^n_s |^\ell.
\end{align*}
Using Lemma \ref{lem:convergence:DHn:control_1}, we choose 
$$\mu=1 + 3(\ell-1)+  \ell \kappa_1$$
such that 
    \begin{align} \nonumber
    & e^{\mu t}|\Delta^n_t|^\ell+ \int_t^T e^{\mu s}|\Delta^n_s|^\ell ds +\dfrac{\ell(1\wedge (\ell-1))}{2} \int_t^T e^{\mu s}|\Delta^n_s|^{\ell-2} \mathbf 1_{\Delta^n_s\neq 0}(D_\theta Z^H_s - D_\theta \mathcal Z^n_s)^2 ds \\ \nonumber
    & \quad \leq  \left(   \dfrac{\eta^\star}{n}\right)^{\ell(q-1)}  e^{\mu T}|D_\theta\eta_T|^\ell -  \ell \int_t^T e^{\mu s}|\Delta^n_s|^{\ell-2} \mathbf 1_{\Delta^n_s\neq 0} \Delta^n_s(D_\theta Z^H_s - D_\theta \mathcal Z^n_s) dW_s \\ \nonumber
    & \qquad  + \left(   \dfrac{\eta^\star}{n}\right)^{\ell(q-1)}\int_t^T e^{\mu s} \left| D_\theta b^\eta_s + p  \left(T+ \left(   \dfrac{\eta^\star}{n}\right)^{q-1} \right)  D_\theta \gamma_s \right|^\ell    ds \\  \nonumber
    & \qquad  +  \int_t^T e^{\mu s}\left|\dfrac{\partial G}{\partial \eta}(s,H_s,\eta_s) - \dfrac{\partial G}{\partial \eta}\left(s-\left( \frac{\eta^\star}{n} \right)^{q-1},\mathcal H^n_s,\eta_s\right) \right|^\ell | D_\theta \eta_s|^\ell   ds \\ \label{eq:estim_diff_BSDE_unif}
    &  \qquad +  \int_t^T e^{\mu s}  \left| \dfrac{\partial G}{\partial h}(s,H_s,\eta_s)  - \dfrac{\partial G}{\partial h}\left(s-\left( \frac{\eta^\star}{n} \right)^{q-1},\mathcal H^n_s,\eta_s\right)  \right|^\ell |D_\theta H_s|^\ell  ds.
\end{align}
From Lemma \ref{lem:convergence:DHn:control_2}, with Young's inequality with $\kappa > 1 $ 
\begin{align*}
&  \left|\dfrac{\partial G}{\partial \eta}(s,H_s,\eta_s) - \dfrac{\partial G}{\partial \eta}\left(s-\left( \frac{\eta^\star}{n} \right)^{q-1},\mathcal H^n_s,\eta_s\right) \right|^\ell | D_\theta \eta_s|^\ell \\
& \quad    \leq  \dfrac{\kappa-1}{\kappa}  \left| \dfrac{\partial G}{\partial \eta}(s,H_s,\eta_s) - \dfrac{\partial G}{\partial \eta}\left(s-\left( \frac{\eta^\star}{n} \right)^{q-1},\mathcal H^n_s,\eta_s\right) \right|^{\ell \frac{\kappa}{\kappa-1}} + \dfrac{1}{\kappa} | D_\theta \eta_s|^{\ell \kappa} \\
& \quad    \leq  \dfrac{\kappa-1}{\kappa}  \left| \Upsilon^n_s \left( \dfrac{H_s}{(T-s)} - \dfrac{\mathcal H^n_s}{(T-s+(\eta^\star/n)^{q-1})}\right) 
 \right|^{\ell \frac{\kappa}{\kappa-1}} + \dfrac{1}{\kappa} | D_\theta \eta_s|^{\ell \kappa} \\
 & \quad    \leq C \dfrac{\kappa-1}{\kappa}  \left|  \dfrac{H_s}{(T-s)} - \dfrac{\mathcal H^n_s}{(T-s+(\eta^\star/n)^{q-1})}
 \right|^{\ell \frac{\kappa}{\kappa-1}} + \dfrac{1}{\kappa} | D_\theta \eta_s|^{\ell \kappa}.
\end{align*}
If $\ell < \varrho$, then there exists $\kappa > 1$ such that $\ell \kappa \leq \varrho$. 
With Lemma \ref{lem:convergence_H_n}, our assumption on $D_\theta \eta$, and the dominated convergence theorem, we deduce that if $\ell < \varrho$
$$\lim_{n\to +\infty} \mathbb E\int_{0}^T e^{\mu s}\left|\dfrac{\partial G}{\partial \eta}(s,H_s,\eta_s) - \dfrac{\partial G}{\partial \eta}\left(s-\left( \frac{\eta^\star}{n} \right)^{q-1},\mathcal H^n_s,\eta_s\right) \right|^\ell | D_\theta \eta_s|^\ell ds  = 0.$$

Again with Lemma \ref{lem:convergence:DHn:control_2},
\begin{align*}
& \mathbb E \int_{T-\delta}^T e^{\mu s}  \left| \dfrac{\partial G}{\partial h}(s,H_s,\eta_s)  - \dfrac{\partial G}{\partial h}\left(s-\left( \frac{\eta^\star}{n} \right)^{q-1},\mathcal H^n_s,\eta_s\right)  \right|^\ell |D_\theta H_s|^\ell  ds \\
& \quad \leq 2^{\ell-1} C \mathbb E \int_{T-\delta}^T |\zeta_s|^\ell \left| \dfrac{H_s}{(T-s)} - \dfrac{\mathcal H^n_s}{(T-s+(\eta^\star/n)^{q-1})}\right|^\ell  ds \\
& \qquad  + 2^{\ell-1} C  \left( \frac{\eta^\star}{n} \right)^{\ell(q-1)} \mathbb E \int_{T-\delta}^T  \dfrac{|\zeta_s|^\ell}{(T-s+(\eta^\star/n)^{(q-1)})^\ell}   ds.
\end{align*}
where 
$$\zeta_s=  \mathbb E^{\mathcal{F}_s} \left[\int_s^T  \left( | D_\theta b^\eta_u|  +  |D_\theta \gamma_u| + | D_\theta \eta_u | \right)  du \right] = \mathbb E^{\mathcal{F}_s} \left[\int_s^T  \varpi_u du \right].$$
Thus 
$$|\zeta_s|^\ell \leq  \mathbb E^{\mathcal{F}_s} \left[\int_s^T  |\varpi_u|^\ell du \right].$$
Next we can argue as in the proof of the previous lemma to deduce that 
$$\lim_{n\to +\infty}  \mathbb E \int_{0}^T e^{\mu s}  \left| \dfrac{\partial G}{\partial h}(s,H_s,\eta_s)  - \dfrac{\partial G}{\partial h}\left(s-\left( \frac{\eta^\star}{n} \right)^{q-1},\mathcal H^n_s,\eta_s\right)  \right|^\ell |D_\theta H_s|^\ell  ds = 0$$
Taking the expectation in \eqref{eq:estim_diff_BSDE_unif}, we deduce that 
$$\lim_{n\to +\infty}\sup_{\theta \in [0,T]} \left( \sup_{t\in [0,T]} \mathbb E e^{\mu t}|\Delta^n_t|^\ell + \mathbb E \int_{0}^T e^{\mu s}|\Delta^n_s|^{\ell-2} \mathbf 1_{\Delta^n_s\neq 0}(D_\theta Z^H_s - D_\theta \mathcal Z^n_s)^2 ds  + \mathbb E \int_{0}^T e^{\mu s} |\Delta^n_s|^\ell ds \right) = 0.$$
The Burkholder-Davis-Gundy inequality leads to:
\begin{align*}
   &  \mathbb E \sup_{t\in [0,T]} \left| \int_t^T e^{\mu s}|\Delta^n_s|^{\ell-2} \mathbf 1_{\Delta^n_s\neq 0} \Delta^n_s(D_\theta Z^H_s - D_\theta \mathcal Z^n_s)\Delta^n_s dW_s \right| \\
   & \leq C \mathbb E \left[ \int_{0}^T e^{2 \mu s}|\Delta^n_s|^{2\ell-2} \mathbf 1_{\Delta^n_s\neq 0} (D_\theta Z^H_s - D_\theta \mathcal Z^n_s)^2  ds\right]^{\frac{1}{2}} \\
   & \leq C \mathbb E \left[ \sup_{t \in [0,T]} e^{\mu t} |\Delta^n_t|^\ell \int_{0}^T e^{ \mu s}|\Delta^n_s|^{\ell-2} \mathbf 1_{\Delta^n_s\neq 0} (D_\theta Z^H_s - D_\theta \mathcal Z^n_s)^2  ds\right]^{\frac{1}{2}} \\
   & \leq \dfrac{1}{2} \mathbb E \left[ \sup_{t \in [0,T]} e^{\mu t} |\Delta^n_t|^\ell \right] + \frac{C^2}{2} \mathbb E \left[\int_{0}^T e^{ \mu s}|\Delta^n_s|^{\ell-2} \mathbf 1_{\Delta^n_s\neq 0} (D_\theta Z^H_s - D_\theta \mathcal Z^n_s)^2  ds \right] .
\end{align*} 
With the same arguments as above, we obtain that: 
$$\lim_{n\to +\infty} \sup_{\theta \in [0,T]} \mathbb E \left[ \sup_{t\in [0,T]} e^{\mu t}|\Delta^n_t|^\ell  \right] = 0.$$
To obtain the convergence of the sequence $\Delta Z^n_s = D_\theta Z^H - D_\theta \mathcal Z^n$ for $1 < \ell  < 2$:
\begin{align*}
& \mathbb E \left[ \int_{0}^T |\Delta Z^n_s|^2 ds \right]^{ \ell/2}  =  \mathbb E \left[ \int_0^T\mathbf 1_{\Delta^n_s \neq 0} |\Delta Z^n_s|^2 ds\right]^{\ell/2}  = \mathbb E \left[ \int_0^T \left( \Delta^n_{s} \right)^{2-\ell}\left( \Delta^n_{s} \right)^{\ell-2}\mathbf 1_{\Delta^n_s \neq 0}  |\Delta Z^n_s|^2ds \right]^{\ell/2} \\ \nonumber
& \quad \leq \mathbb E \left[ \left( \Delta^n_{*}\right)^{\ell(2-\ell)/2} \left(\int_0^T \left( \Delta^n_{s} \right)^{\ell-2}\mathbf 1_{\Delta^n_s \neq 0}  |\Delta Z^n_s|^2 ds \right)^{\ell/2}\right]\\ \nonumber
& \quad \leq \left\{\mathbb E \left[ \left( \Delta^n_{*} \right)^{\ell}\right]\right\}^{(2-\ell)/2}  \left\{\mathbb E \int_0^T \left( \Delta^n_{s}\right)^{\ell-2} \mathbf 1_{\Delta^n_s \neq 0} |\Delta Z^n_s|^2 ds \right\}^{\ell/2} \\
& \quad  \leq \frac{2-\ell}{2} \mathbb E \left[ \left( \Delta^n_{*} \right)^{\ell}\right] + \frac{\ell}{2} \mathbb E \int_0^T \left( \Delta^n_{s} \right)^{\ell-2} \mathbf 1_{\Delta^n_s \neq 0} |\Delta Z^n_s|^2 ds
\end{align*}
where we have used H\"older's and Young's inequality with $ \frac{2-\ell}{2} + \frac{\ell}{2}=1$ and  $\Delta^n_{*} = \sup_{s \in [0,T]} |\Delta^n_s|$. This achieves the proof of the lemma. 
\end{proof}

\subsection{Convergence of $D_\theta Y^n$ and complete proof of Theorem \ref{thm:main_result}}

Let us start with the convergence for fixed parameters $\theta$ and $t$. 
\begin{proposition}
If we set $D_\theta Y_T= 0$ for $0\leq \theta \leq T$, then for any $0\leq \theta \leq t \leq T$, we have the almost surely convergence:
$$\lim_{n\to +\infty} D_\theta Y_t^n = D_\theta Y_t.$$
\end{proposition}

\begin{proof}
    For any $0\leq \theta \leq t < T$, according to the first part of Theorem \ref{thm:main_result}, Proposition \ref{propo:DYn:Deta:DHn} and Lemma \ref{lem:convergence:DHn},
    $$D_\theta Y^n_t= \dfrac{1}{\left(T-t+ \left(\frac{\eta^\star}{n}\right)^{q-1}\right)^{p-1}} D_\theta \eta_t + \dfrac{1}{\left(T-t+ \left(\frac{\eta^\star}{n}\right)^{q-1}\right)^{p}}D_\theta \mathcal H^n_t$$
converges a.s. to 
    $$ \dfrac{D_\theta \eta_t}{(T-t)^{p-1}} + \dfrac{1}{(T-t)^{p}} D_\theta H_t= D_\theta Y_t.$$
For $t = T$, since $Y^n_T = n$, $D_\theta Y^n_T = 0$ for any $n$. Another way to obtain this fact consists of using Proposition \ref{propo:DYn:Deta:DHn} and the terminal condition of the process $\mathcal{H}^n$:
$$D_\theta Y^n_T = \dfrac{n}{\eta^\star} D_\theta \eta_T +\left(  \dfrac{n}{\eta^\star} \right)^q D_\theta \mathcal{H}^n_T =  \dfrac{n}{\eta^\star} D_\theta \eta_T - \left(  \dfrac{n}{\eta^\star} \right)^q (\eta^\star)^q \dfrac{1}{n^{q-1}}\dfrac{D_\theta \eta_T}{\eta^\star} = 0. $$
Hence the sequence $D_\theta Y^n_T$ converges a.s. to $0=D_\theta Y_T$. 
\end{proof}

Note that 
\begin{align*}
D_\theta H_t- D_\theta \mathcal H^n_t & = (T-t)^{p} D_\theta Y_t -\left(T-t+ \left(\frac{\eta^\star}{n}\right)^{q-1}\right)^{p} D_\theta Y^n_t + \left(\frac{\eta^\star}{n}\right)^{q-1} D_\theta \eta_t \\
& = (T-t)^{p} (D_\theta Y_t -D_\theta Y^n_t ) - \left[  \left(T-t+ \left(\frac{\eta^\star}{n}\right)^{q-1}\right)^{p} -  (T-t)^{p} \right] D_\theta Y^n_t + \dfrac{\eta^\star}{n}D_\theta \eta_t. 
\end{align*}
This equality allows us to complete the statement and the proof of Theorem \ref{thm:main_result}. 
Indeed from the previous remark, 
\begin{align*}
 (T-t)^{p} (D_\theta Y_t -D_\theta Y^n_t ) & = D_\theta H_t- D_\theta \mathcal H^n_t +  \left[  \left(T-t+ \left(\frac{\eta^\star}{n}\right)^{q-1}\right)^{p} -  (T-t)^{p} \right] D_\theta Y^n_t - \dfrac{\eta^\star}{n}D_\theta \eta_t. 
\end{align*}
Furthermore with Proposition \ref{propo:DYn:Deta:DHn}
\begin{align*}
&  \left[  \left(T-t+ \left(\frac{\eta^\star}{n}\right)^{q-1}\right)^{p} -  (T-t)^{p} \right] D_\theta Y^n_t   = p \left(\frac{\eta^\star}{n}\right)^{q-1}D_\theta Y^n_t  \int_0^1   \left(T-t+ a \left(\frac{\eta^\star}{n}\right)^{q-1}\right)^{p-1} da \\
& \quad = p \left(\frac{\eta^\star}{n}\right)^{q-1}D_\theta \eta_t  A^n_t + p \left(\frac{\eta^\star}{n}\right)^{q-1} \dfrac{1}{\left(T-t+ \left(\frac{\eta^\star}{n}\right)^{q-1}\right)}D_\theta \mathcal H^n_t  A^n_t \\
& \quad =  p \left(\frac{\eta^\star}{n}\right)^{q-1}D_\theta \eta_t  A^n_t + p \left(\frac{\eta^\star}{n}\right)^{q-1} \dfrac{1}{\left(T-t+ \left(\frac{\eta^\star}{n}\right)^{q-1}\right)} (D_\theta \mathcal H^n_t - D_\theta H_t)  A^n_t \\
& \qquad +  p \left(\frac{\eta^\star}{n}\right)^{q-1} \dfrac{1}{\left(T-t+ \left(\frac{\eta^\star}{n}\right)^{q-1}\right)} D_\theta H_t  A^n_t 
\end{align*}
with 
$$0\leq A^n_t =\dfrac{1}{\left(T-t+ \left(\frac{\eta^\star}{n}\right)^{q-1}\right)^{p-1}}  \int_0^1   \left(T-t+ a \left(\frac{\eta^\star}{n}\right)^{q-1}\right)^{p-1} da \leq 1.$$
Note that for any $n$ and $t$
$$0\leq \left(\frac{\eta^\star}{n}\right)^{q-1} \dfrac{1}{\left(T-t+ \left(\frac{\eta^\star}{n}\right)^{q-1}\right)}\leq 1$$
and from \eqref{eq:estim_Mallia_deriv_H}  for  $t \in [0,T]$, 
$$  |D_\theta H_t| \leq C (T-t) \mathbb E^{\mathcal{F}_t} \left[\int_t^T  \left( | D_\theta b^\eta_s|  +  |D_\theta \gamma_s| + | D_\theta \eta_s | \right)  ds \right].
$$
Therefore 
\begin{align*}
 \left[  \left(T-t+ \left(\frac{\eta^\star}{n}\right)^{q-1}\right)^{p} -  (T-t)^{p} \right] |D_\theta Y^n_t| &  \leq  p \left(\frac{\eta^\star}{n}\right)^{q-1} |D_\theta \eta_t|  + p |D_\theta \mathcal H^n_t - D_\theta H_t| \\
&  +  p \left(\frac{\eta^\star}{n}\right)^{q-1} \mathbb E^{\mathcal{F}_t} \left[\int_t^T  \left( | D_\theta b^\eta_s|  +  |D_\theta \gamma_s| + | D_\theta \eta_s | \right) ds \right].
\end{align*}
In other words 
\begin{align*}
 (T-t)^{p} |D_\theta Y_t -D_\theta Y^n_t | & \leq (1+p) |D_\theta H_t- D_\theta \mathcal H^n_t | +  \left( \dfrac{\eta^\star}{n} +p \left(\frac{\eta^\star}{n}\right)^{q-1}  \right)  |D_\theta \eta_t| \\
 & \quad +  p \left(\frac{\eta^\star}{n}\right)^{q-1} \mathbb E^{\mathcal{F}_t} \left[\int_t^T  \left( | D_\theta b^\eta_s|  +  |D_\theta \gamma_s| + | D_\theta \eta_s | \right)  ds \right].
 \end{align*}
 The conclusion directly comes from Lemma \ref{lem:unif_convergence_DHn}.

%

\section{Applications and Examples} \label{sect:applications}

 \subsection{Gradient of the related PDE} \label{ssect:gradient}

Here we consider that $\eta$ and $\gamma$ are smooth functions, $\eta_t=\eta(t,X_t)$ and $\gamma_t = \gamma(t,X_t)$, of the solution $X=X^x$ of the SDE: for $x \in \mathbb R^d$
\begin{align} \label{eq:sde}
    X_t & = x + \int_0^t b(s,X_s) ds + \int_0^t \sigma(s,X_s) dW_s, \quad 0\leq t\leq T.
\end{align}
We suppose that $b$ and $\sigma$ are continuous on $[0,T] \times \R^d$ and of class $C^1$ with respect to $x$ with bounded first derivatives. According to \cite[Theorem 2.2.1]{nual:06}, we have :
\begin{lemme} \label{lem:derive_eds} 
The SDE admits a unique solution $X$ in $S^\infty(0,T)$ such that:
\begin{enumerate}
    \item For all $t\in [0,T]$, $X^i_t \in \mathbb{D}^{1,\infty}$ and for all $p\in [1,+\infty[$, \begin{equation} \label{eq:estim_deriv_X}
    \underset{0\leq \theta \leq t}{\sup} \mathbb{E}\left[\underset{\theta \leq s \leq T}{\sup} |D_\theta X_s^i|^p\right] < +\infty.
    \end{equation}
    \item The process $DX^i$ satisfies the linear SDE 
$$D_\theta X^i_t = \sigma_i (\theta, X_\theta) + \sum_{k=1}^d \int_\theta^t \frac{\partial b_i}{\partial x_k}(s,X_s) D_\theta X^k_s ds+ \sum_{j=1}^d \sum_{k=1}^d \int_\theta^t \frac{\partial \sigma_i^j}{\partial x_k}(s,X_s) D_\theta X^k_s dW^j_s.$$
\end{enumerate}
\end{lemme}
Under this setting, 
$$\eta_t = \eta(t,X_t) = \eta(0,x) + \int_0^t (\mathcal L \eta) (s,X_s) ds + \int_0^t \partial_x \eta(s,X_s) \sigma(s,X_s) dW_s,$$
where $\mathcal L$ is the infinitesimal generator of the SDE \eqref{eq:sde}:
\begin{equation} \label{eq:infi_gene_X}
\mathcal{L} \phi = \langle b,\partial_x  \phi \rangle + \dfrac{1}{2} \text{tr}\parent{\sigma  \sigma^*\partial^2_x \phi}.
\end{equation}
Hence $b^\eta_s = (\mathcal L \eta) (s,X_s)$ and $\sigma^{\eta}_s = \partial_x\eta(s,X_s) \sigma(s,X_s)$. 
\begin{example} \label{example:diff_case}
If we apply It\^o's formula to $\eta=\varphi(X)$, with $d=1$ and $\varphi(x) = \dfrac{\eta^\star-\eta_\star}{\pi} \arctan(x) + \dfrac{\eta^\star+\eta_\star}{2}$, then $\eta$ satisfies all required conditions (Assumptions \ref{assump:eta_gamma} and \ref{assumption:Malliavin:b:eta:gamma}) of the previous sections.
\end{example}

\medskip 
In this section, the superscript $t,x$ indicates the dependence of the solution on the initial data $(t,x)$, and it will be omitted when the context is clear. 
In this Markovian setting, it is known that the coupled system of Equations \eqref{eq:sde}-\eqref{eq:liquid_BSDE} is related to the solution of the PDE:
\begin{equation} \label{eq:PDE}
    \dfrac{\partial u}{\partial t} + \mathcal L u - (p-1) \dfrac{|u|^{q-1}}{\eta(t,x)^{q-1}} u + \gamma(t,x) = 0, \quad u(T,\cdot) = +\infty
\end{equation}

If $Y^{t,x}$ solves the BSDE \eqref{eq:liquid_BSDE} when $\eta_s$ and $\gamma_s$ are replaced by $\eta(s,X^{t,x}_s)$ and $\gamma(s,X^{t,x}_s)$, then for any $0 \leq t\leq s < T$, $Y^{t,x}_s = u(s,X^{t,x}_s)$ and $u$ is the viscosity solution of the previous PDE. 
See \cite{popi:06,grae:hors:sere:18,popi:17,caci:deni:popi:25}. Furthermore the expansion \eqref{eq:asymp_Y} of $Y$ corresponds to:
\begin{equation} \label{eq:expansion_u}
\forall (t,x) \in [0,T) \times \mathbb R^d, \qquad u(t,x) = \dfrac{\eta(t,x)}{(T-t)^{p-1}} + \dfrac{h(t,x)}{(T-t)^{p}},
\end{equation}
with $|h(t,x)| = O(T-t)^2$ at time $T$. This property is also given in \cite[Lemma 4.1]{grae:hors:sere:18}. Here $h$ is defined thanks to the relation: $h(s,X^{t,x}_s) = H^{t,x}_s$. 

\begin{proposition} \label{prop:regul_solution_PDE}
    We assume that: 
    \begin{itemize}
        \item The functions $(t,x) \mapsto \eta(t,x)$, $(t,x) \mapsto \mathcal L \eta(t,x)$ and $(t,x) \mapsto \gamma(t,x)$ are continuous w.r.t. $(t,x)$ and of class $C^1$ w.r.t. $x$, with bounded derivatives.
        \item The matrix $\sigma$ is uniformly elliptic, that is if there exists $\lambda > 0$ such that  $$\forall s \in [0,T],\ \forall (x,y)\in \mathbb R^d \times \R^d, \quad \langle \sigma(s,x) \sigma^*(s,x) y,y \rangle \geq \lambda |y|^2.$$
    \end{itemize}
    The solution $u$ is of class $C^1$ w.r.t. $x$ and for $0 \leq t < T$ and $x \in \mathbb R^d$
    $$\partial_x u(t,x) = \dfrac{\partial_x \eta(t,x)}{(T-t)^{p-1}} + \dfrac{\partial_x h(t,x)}{(T-t)^{p}}.$$
    Finally the approximating sequences $u^n$ and $\partial_x u^n$ (resp. $h^n$ and $\partial_x h^n$) pointwise converge to $u$ and $\partial_x u$ on $[0,T)\times \mathbb R^d$ (resp. to $h$ and $\partial_x h$ on $[0,T] \times \mathbb R^d$), where $u^n$ is the continuous viscosity solution of the PDE \eqref{eq:PDE} with terminal condition $u^n(T,\cdot) = n$ and $h^n$ is defined by 
    \begin{equation*} 
        u^n(t,x) = \dfrac{\eta(t,x)}{\left(T-t+ \left(\frac{\eta^\star}{n}\right)^{q-1}\right)^{p-1}}  + \dfrac{h^n(t,x)}{\left(T-t+ \left(\frac{\eta^\star}{n}\right)^{q-1}\right)^{p}}, \quad (t,x) \in [0,T] \times \mathbb{R}^d.
    \end{equation*}
\end{proposition}

\begin{proof}
The function $h^n$ is linked to $\mathcal H^n$ by the relation: $h^n(s,X^{t,x}_s) = \mathcal H^{n,t,x}_s$. We know that $u^n$ and $h^n$ pointwise converge to $u$ and $h$ (from the convergence of $Y^n$ and $\mathcal H^n$ to $Y$ and $H$), which gives the asymptotic expansion of $u$.

Under our setting, we use \cite[Lemma 2.4 and Theorem 3.1]{ma:zhang:2002} to deduce that $D_\theta Y^n_s$ and $D_\theta \mathcal H^n_s$ are of the form
\begin{align*}
D_\theta Y^n_s = \nabla Y^n_s (\nabla X_\theta)^{-1} \sigma(\theta,X_\theta) \mathbf 1_{\theta \leq s} ,\quad &  D_\theta \mathcal H^n_s = \nabla \mathcal H^n_s (\nabla X_\theta)^{-1} \sigma(\theta,X_\theta) \mathbf 1_{\theta \leq s}, 
\end{align*}
where the processes $\nabla X$, $\nabla Y^n$ and  $\nabla \mathcal H^n$ are solution of the variational equations related to \eqref{eq:sde},  \eqref{edsr:tronquee} and \eqref{eq:BSDE_H_n}. Furthermore $h^n$ is of class $C^1$ w.r.t. $x$  with $\partial_x h^n = \nabla \mathcal H^n$ and $\partial_x h^n$ is continuous on $[0,T] \times \mathbb R^d$. Let us emphasize that these results cannot be directly used for $Y$ (singularity in the terminal condition) or $H$ (singularity in the generator). 

However we can define the solution $(\nabla H,\nabla Z^H)$ of the variational equation related to \eqref{eq:BSDE_H}
    \begin{align*} \nonumber 
    \nabla H_s & = \int_s^T  \left[ (T-u) \partial_x b^\eta(u,X_u) \nabla X_u  + (T-u)^p \partial_x \gamma(u,X_u) \nabla X_u \right] du \\
    & -  \int_s^T  \dfrac{\partial G}{\partial \eta}(u,H_u,\eta_u) \partial_x \eta(u,X_u) \nabla X_u du    - \int_s^T  \dfrac{\partial G}{\partial h}(u,H_u,\eta_u)  \nabla H_u du - \int_s^T \nabla Z^H_u dW_u
    \end{align*}
 using again that a.s. $|H_u | \leq C(T-u)^2$ close to $T$. And from Lemma \ref{lem:mallia_deriv_H}, we also have 
$$D_\theta H_s = \nabla H_s(\nabla X_\theta)^{-1} \sigma(\theta,X_\theta) \mathbf 1_{\theta \leq s}.$$
Then we check the proof of \cite[Theorem 3.1]{ma:zhang:2002} to deduce that $\partial_x h$ exists. We define 
$$\nabla X^{\varepsilon}_s = \dfrac{1}{\varepsilon} \left( X^{t,x+\varepsilon}_s - X^{t,x}_s \right),\quad \nabla H^{\varepsilon}_s = \dfrac{1}{\varepsilon} \left( H^{t,x+\varepsilon}_s - H^{t,x}_s \right),\quad \nabla Z^{\varepsilon}_s = \dfrac{1}{\varepsilon} \left( Z^{H,t,x+\varepsilon}_s - Z^{H,t,x}_s \right)$$
and prove that $\nabla H^{\varepsilon}$ converges to $\nabla H$ when $\varepsilon$ goes to zero. First note that 
    \begin{align*} \nonumber 
    \nabla H^\varepsilon_s & = \int_s^T  \left[ (T-u) \partial_x \widetilde b^\eta(u)  + (T-u)^p \partial_x  \widetilde \gamma(u) - \partial_\eta \widetilde G(u) \right] \nabla X^\varepsilon_u  du - \int_s^T  \partial_h \widetilde G(u)  \nabla H^\varepsilon_u du - \int_s^T \nabla Z^\varepsilon_u dW_u
    \end{align*}
where 
\begin{align*}
\partial_x \widetilde b^\eta (u) & = \int_0^1 \partial_x b^\eta(u,X^{t,x}_u + a (X^{t,x+\varepsilon}_u-X^{t,x}_u)) da \\
\partial_x  \widetilde \gamma(u)  & =  \int_0^1 \partial_x \gamma (u,X^{t,x}_u + a (X^{t,x+\varepsilon}_u-X^{t,x}_u)) da \\
\partial_\eta \widetilde G(u) & =  \int_0^1 \partial_\eta G (u,\eta^{t,x}_u + a (\eta^{t,x+\varepsilon}_u-\eta^{t,x}_u),H^{t,x+\varepsilon}_u) da \times  \int_0^1 \partial_x \eta (u,X^{t,x}_u + \alpha (X^{t,x+\varepsilon}_u-X^{t,x}_u)) d\alpha \\
 \partial_h \widetilde G(u) & =  \int_0^1 \partial_h G (u,\eta^{t,x}_u ,H^{t,x}_u+ a (H^{t,x+\varepsilon}_u-H^{t,x}_u)) da .
\end{align*}
The key point now is that $|H^{t,x}_u| \leq C(T-u)^2$ on the interval $[T-\delta,T]$ and both constants $C$ and $\delta$ depend only on the bounds on $\eta$ and $\gamma$, and not on $x$. Therefore the processes $\partial_\eta \widetilde G $ and $\partial_h \widetilde G$ are bounded, uniformly w.r.t. $\varepsilon$. The rest of the proof can be copied from \cite{ma:zhang:2002}, to conclude that $\partial_x h$ exists and is equal to $\nabla H$ and that $\partial_x h$ is continuous on $[0,T] \times \mathbb R^d$.

From \eqref{eq:expansion_u}, we deduce that $\partial_x u$ exists on $[0,T) \times \mathbb R^d$ and is given by the statement of the proposition. From our convergence result (Lemma \ref{lem:unif_convergence_DHn}), we deduce that the sequence $\nabla \mathcal H^n$ converges to $\nabla H$. Hence the sequence $\partial_x h^n$ converges to $\partial_x h$ on $[0,T]\times \mathbb R^d$. The result follows immediately for $\partial_x u^n$. 
\end{proof}

Estimate \eqref{eq:estim_Mallia_deriv_H} becomes for $0\leq \theta \leq t\leq T$
\begin{align*}
  |D_\theta H_t| & = |\nabla H_t  (\nabla X_\theta)^{-1} \sigma(\theta,X_\theta) | \\
  & \leq C (T-t) \mathbb E^{\mathcal{F}_t} \left[\int_t^T  \left( | \partial_x (\mathcal L u)(s,X_s) |  +  |\partial_x\eta(s,X_s)| + | \partial_x\gamma(s,X_s)| \right) |D_\theta X_s| ds \right] \\
  & \leq C (T-t) \mathbb E^{\mathcal{F}_t} \left[\int_t^T |D_\theta X_s| ds \right].
\end{align*}
From Lemma \ref{lem:derive_eds} and H\"older's inequality, we deduce that 
$$|\partial_x h(t,x) | =  |\nabla H_t| \leq C (T-t).$$

\begin{remarque}
In \cite[Theorem 2.9]{grae:hors:sere:18}, it is already proved that $u$ is of class $C^1$ w.r.t. $t$. 
\end{remarque}

\subsection{Sensitivity in liquidation problem}

Malliavin calculus is a useful tool to analyze the sensitivity in finance, see among many others \cite{four:lasr:99,gobe:muno:05}. In the liquidation problem mentioned in the introduction, the optimal state process is given by 
$$\Xi_s = x \exp\parent{-\int_t^s \parent{\dfrac{Y_u}{\eta_u}}^{q-1} du}$$
or with the previous notations:
\begin{align*}
\Xi_s & = x \exp\parent{-\int_t^s  \dfrac{1}{(T-u)} \parent{ 1+ \dfrac{H_u}{\eta_u(T-u)}}^{q-1} du} \\
& = x\dfrac{T-s}{T-t} \exp\parent{-\int_t^s  \dfrac{1}{(T-u)} \left[ \parent{ 1+ \dfrac{H_u}{\eta_u(T-u)}}^{q-1} -1 \right]du} .
\end{align*}
In particular for $0\leq \theta \leq s < T$
\begin{align} \label{eq:Malliavin_deriv_Xi}
D_\theta \Xi_s & = -(q-1) \Xi_s \int_t^s  \left| \dfrac{Y_u}{\eta_u} \right|^{q-2} \mbox{sign }(Y_u) D_\theta \parent{\dfrac{Y_u}{\eta_u}}  du .
\end{align}
A key argument in the greeks computations is the positivity of the Malliavin covariance matrix. This property is ensured if there is a diffusion part in $\Xi$ with an elliptic diffusion matrix (see again \cite{four:lasr:99,gobe:muno:05}). In our case, there is no diffusion part for $\Xi$. Worse than this, we know some degenerate examples. 

\medskip 

Indeed in \cite[Section 5]{anki:jean:krus:13}, the authors consider the case where $\gamma=0$ and $\eta$ has uncorrelated multiplicative increments. In our setting, it means that:
\begin{lemme}
When $\eta$ is an It\^o process, $\eta$ has uncorrelated multiplicative increments if and only if the drift $b^\eta$ is of the form: $b^\eta_t = g(t) \eta_t$ where $g$ is a deterministic function. 
\end{lemme}
\begin{proof}
If the drift $b^\eta$ is of the form: $b^\eta_t = g(t) \eta_t$, from \cite[Example 5.1]{anki:jean:krus:13}, $\eta$ has uncorrelated multiplicative increments. Conversely we have 
$$\eta_t = \eta_0 + \int_0^t b^\eta_s ds + \int_0^t \sigma_s dW_s \Longrightarrow \mathbb E \eta_t = \eta_0 + \int_0^t \mathbb E (b^\eta_s) ds.$$
If we consider $M_t = \eta_t / \mathbb E \eta_t $, from \cite[Lemma 5.1]{anki:jean:krus:13}, $\eta$ has uncorrelated multiplicative increments if and only if $M$ is a martingale. But 
$$dM_t = \dfrac{1}{(\mathbb E \eta_t)^2} \left[  \mathbb E (\eta_t ) b^\eta_t - \mathbb E (b^\eta_t) \eta_t \right] dt + \dfrac{\sigma_t}{\mathbb E \eta_t} dW_t.$$
Hence $b^\eta_t = \left(  \dfrac{ \mathbb E b^\eta_t }{\mathbb E \eta_t} \right) \eta_t $.
\end{proof}
From \cite[Propositions 5.2 and 5.3]{anki:jean:krus:13}, we know that $\eta$ has uncorrelated multiplicative increments if and only if $\Xi$ is deterministic, that is $D_\theta \Xi_s = 0$ for any $0\leq \theta \leq s \leq T$. We also have another result.
\begin{lemme}
When $\gamma = 0$, $\eta$ has uncorrelated multiplicative increments if and only if for any $0\leq \theta \leq u < T$, $D_\theta \parent{\dfrac{Y_u}{\eta_u}}= 0$.
\end{lemme}
\begin{proof}
From the proof of \cite[Proposition 5.3]{anki:jean:krus:13}, we also know that if $\eta$ has uncorrelated multiplicative increments, then $\dfrac{Y}{\eta}$ is deterministic, hence its Malliavin derivative is zero. Conversely if the Malliavin derivaitive of $Y/\eta$ is null, using \eqref{eq:Malliavin_deriv_Xi}, $D_\theta \Xi_s=0$, thus $\Xi$ is deterministic and $\eta$ has uncorrelated multiplicative increments. 
\end{proof}
Moreover we can explicitly compute $H$, either from \cite[Proposition 5.3]{anki:jean:krus:13} where $Y$ is given, or directly. 
\begin{lemme}
Assume that $\eta$ has uncorrelated multiplicative increments, with $b^\eta_t = g(t) \eta_t$. Then $H_t = \eta_t (T-t) h(t)$ with
$$h(t) =   -1 + \left( \dfrac{1}{T-t} \int_t^T  \exp\left( - (q-1) \int_t^s g(u)du  \right) ds \right)^{1-p} .$$
\end{lemme}
\begin{proof}
Since $\gamma= 0$ and $b^\eta_t = g(t) \eta_t$, $H$ is the solution of the BSDE with generator
\begin{align*}
F(t,h) & = (T-t) \eta_t g(t)  - (p-1)\eta_t \left[ \left( 1 + \dfrac{1}{\eta_t (T-t)} h\right) \left| 1 + \dfrac{1}{\eta_t (T-t)} h\right|^{q-1} -1 - q\dfrac{1}{\eta_t (T-t)} h \right]
\end{align*}
and terminal condition $0$. Make the ansatz that $H_t = \eta_t (T-t) h(t)$. Then
\begin{align*}
dH_t & = \left[ (T-t) h(t) \eta_t g(t) +\eta_t (T-t)h'(t) -\eta_t  h(t) \right] dt + \sigma^\eta_t dW_t, \\
F(t,H_t) & = (T-t) \eta_t g(t)  - (p-1)\eta_t \left[ \left( 1 + h(t) \right) \left| 1 +h(t) \right|^{q-1} -1 - qh(t) \right].
\end{align*}
Therefore $dH_t + F(t,H_t)dt $ is a martingale if 
$$(T-t) i(t) g(t) +  (T-t)i'(t)  - (p-1) \left[ i(t) \left| i(t) \right|^{q-1} -i(t)\right] = 0$$
with $i(t) = h(t) +1$. Define $G$ as the solution of $G'(t) = g(t) G(t)$ with $G(0)=1$. Then 
$$ (T-t) (i(t) G(t))'  = (p-1) G(t) \left[ i(t) \left| i(t) \right|^{q-1} - i(t)\right].$$
We can verify that: 
$\displaystyle i(t) G(t) = \left( \dfrac{1}{T-t} \int_t^T \dfrac{1}{G(s)^{q-1}} ds \right)^{1-p}.$
Thus 
$$h(t) = -1 + \left( \dfrac{1}{T-t} \int_t^T \left( \dfrac{G(t)}{G(s)} \right)^{q-1} ds \right)^{1-p}.$$
Note that $h(t) \sim (T-t)$ as $t$ goes to $T$. By uniqueness of $H$, we obtain the result. 
\end{proof}
Remark that we obtain an explicit expression for $Y$:
\begin{align*}
Y_t & = \dfrac{\eta_t}{(T-t)^{p-1}} + \dfrac{H_t}{(T-t)^p} =  \dfrac{\eta_t}{(T-t)^{p-1}} \left( \dfrac{1}{T-t} \int_t^T  \exp\left( - (q-1) \int_t^s g(u)du  \right) ds \right)^{1-p} \\
& =  \eta_t \left( \int_t^T \exp\left( - (q-1) \int_t^s g(u)du  \right) ds \right)^{1-p} .
\end{align*}

Let us consider the case where $\eta$ is deterministic. Then for $\theta \leq s < T$
\begin{align*}
D_\theta \Xi_s &  = -(q-1) \Xi_s \int_t^s \dfrac{1}{(T-u)^2 \eta_u}  \left| 1+ \dfrac{H_u}{\eta_u(T-u)} \right|^{q-2}   \mbox{sign }\parent{1+ \dfrac{H_u}{\eta_u(T-u)}} D_\theta H_udu.
\end{align*}
and $D_\theta H$ is the solution of the BSDE \eqref{eq:BSDE_Malliavin_deriv_H}, which becomes:
\begin{align*}
D_\theta H_t & = \int_t^T  \left[ (T-s)^{p} D_\theta \gamma_s  - \dfrac{\partial G}{\partial h}(s,H_s,\eta_s)  D_\theta H_s \right]ds - \int_t^T D_\theta Z^H_s dW_s = \mathbb E^{\mathcal{F}_t} \left[ \int_t^T (T-s)^{p} D_\theta \gamma_s \Gamma_{t,s} ds \right]
\end{align*}
where $\Gamma_{t,s}$ is given by \eqref{eq:def_Gamma}. If $\gamma_s = \gamma(s,X_s)$, then $D_\theta \gamma_s = (\partial_x \gamma) (s,X_s) D_\theta X_s$. In this case, we can find easy conditions on $\partial_x \gamma$ and the coefficients $b$ and $\sigma$ of the SDE of $X$ such that the Malliavin covariance matrix of $D_\theta H$ is definite positive. For example if $\partial_x \gamma$ is bounded away from zero and if the parameters $b$ and $\sigma$ satisfy the conditions of Lemma \ref{lem:derive_eds} and  if $\sigma$ is uniformly elliptic (see Proposition \ref{prop:regul_solution_PDE} for the definition), we can apply \cite[Theorems 2.3.1 and 2.3.3]{nual:06}.

\medskip 

Now in general we have for $0 \leq \theta \leq s < T$
\begin{align*}
D_\theta \Xi_s & = -(q-1) \Xi_s \int_t^s  \left| \dfrac{Y_u}{\eta_u} \right|^{q-2} \mbox{sign }(Y_u) D_\theta \parent{\dfrac{Y_u}{\eta_u}}  du \\
& = -(q-1) \Xi_s \int_t^s \dfrac{1}{(T-u)^2 \eta_u^2}  \left| 1+ \dfrac{H_u}{\eta_u(T-u)} \right|^{q-2}   \mbox{sign }\parent{1+ \dfrac{H_u}{\eta_u(T-u)}} \left[ \eta_u D_\theta H_u - H_u D_\theta \eta_u \right]  du
\end{align*}
with $D_\theta H$ given by \eqref{eq:BSDE_Malliavin_deriv_H}. Existence of tractable conditions such that the Malliavin covariance matrix is definite positive is left for further research. 

\section{Appendix}

Let us evoke the arguments of \cite[Theorem 23]{grae:popi:21}. We want to solve the BSDE \eqref{eq:BSDE_H}
\begin{equation*}
H_t = \int_t^T F(s,H_s) ds - \int_t^T Z^H_s dW_s = \mathbb E^{\mathcal{F}_t} \left[  \int_t^T F(s,H_s) ds \right]
\end{equation*}
where $F$ is given by \eqref{eq:gene_F}. 
We define the operator 
$$\Gamma(H)_t =  \mathbb E^{\mathcal{F}_t} \left[  \int_t^T F(s,H_s) ds \right]$$
and a solution is a fixed point of this operator $\Gamma$. 
\begin{proposition} \label{prop:existence_H}
If $\eta$ is bounded away from zero by $\eta_\star$ and if the drift $b^\eta$ of $\eta$ is bounded, 
there exists a process $(H,Z^H)$ solution of the previous BSDE and there exists three constants $\delta > 0$, $R > 0$ and $C$ such that a.s. on $[T-\delta,T]$, $|H_t | \leq R (T-t)^2$ and on $[0,T]$, $|H_t|\leq C$. 
\end{proposition}
\begin{proof}
Remark that 
\begin{align*}
G(t,h,\eta)& = (p-1)\eta \left[\left( 1 + \dfrac{1}{\eta (T-t)} h\right) \left| 1 + \dfrac{1}{\eta (T-t)} h\right|^{q-1} -1 - q\dfrac{1}{\eta (T-t)} h \right]\\
& = \dfrac{q h^2}{\eta (T-t)^2} \int_0^1\left| 1 + a\dfrac{1}{\eta (T-t)} h\right|^{q-2} \mbox{sign } \left(1 + a\dfrac{1}{\eta (T-t)} h\right) (1-a) da.
\end{align*}
Now suppose for a while that $|H_t| \leq R(T-t)^2$ on $[T-\delta , T]$. Then for any $a \in [0,1]$ and $T-\delta \leq t \leq T$
$$a\dfrac{|H_t|}{\eta_t (T-t)} \leq a \dfrac{R (T-t)}{\eta_t} \leq \dfrac{R \delta}{\eta_\star}\leq \dfrac{1}{2},$$
if we choose $\delta \leq \dfrac{\eta_\star}{2R}$. Moreover 
\begin{align*}
  \dfrac{\partial G}{\partial h}(t,h,\eta)& = \dfrac{p}{ (T-t)} \left( \left| 1 + \dfrac{1}{\eta (T-t)} h\right|^{q-1} -  1 \right) \\
  & = \dfrac{qh}{\eta (T-t)^2} \int_0^1\left| 1 + a\dfrac{1}{\eta (T-t)} h\right|^{q-2} \mbox{sign } \left(1 + a\dfrac{1}{\eta (T-t)} h\right)  da.
  \end{align*}
Thus if again $|H_t| \leq R(T-t)^2$ on $[T-\delta , T]$, and under our previous condition on $\delta$, we have 
\begin{align*}
 \left| \dfrac{\partial G}{\partial h}(t,H_t,\eta_t) \right|& \leq \dfrac{qR}{\eta_\star} 2^{|q-2|} = L.
  \end{align*}
Therefore if both $H$ and $\widetilde H$ are bounded from above by $t\mapsto R(T-t)^2$ on $[T-\delta , T]$, then 
\begin{align*}
& | \Gamma(H)_t - \Gamma(\widetilde H)_t | \leq  \mathbb E^{\mathcal{F}_t} \left[  \int_t^T |F(s,H_s) -F(s,\widetilde H_s)| ds \right] \leq  \mathbb E^{\mathcal{F}_t} \left[  \int_t^T  L |H_s - \widetilde H_s| ds \right]
\\ &\quad  \leq \mathbb E^{\mathcal{F}_t} \left[  \int_t^T  L  \dfrac{|H_s - \widetilde H_s|}{(T-s)^2} (T-s)^2 ds \right] \leq L(T-t)^3 \|H-\widetilde H\|_{\mathcal H^\delta} \\
& \quad \leq \delta L \|H-\widetilde H\|_{\mathcal H^\delta}  (T-t)^2 \leq \dfrac{1}{2} \|H-\widetilde H\|_{\mathcal H^\delta}  (T-t)^2
\end{align*}
if $\delta \leq 1/(2L)$. Hence 
$$\|\Gamma(H) - \Gamma(\widetilde H)\|_{\mathcal H^\delta} \leq \dfrac{1}{2 } \|H-\widetilde H\|_{\mathcal H^\delta}$$
that is $\Gamma$ is a contraction on the ball of $\mathcal H^\delta$ with radius $R$. Finally 
\begin{align*}
& | \Gamma(H)_t  | \leq  | \Gamma(H)_t  - \Gamma(0)_t | + |\Gamma(0)_t| \leq  \dfrac{1}{2 } \|H\|_{\mathcal H^\delta} (T-t)^2 + \mathbb E^{\mathcal{F}_t} \left[ \int_t^T \left|(T-s) b^\eta_s  + (T-s)^{p} \gamma_s \right| ds \right]   \\
&\quad  \leq \dfrac{R}{2} (T-t)^2 + (T-t)^2  \left[ \dfrac{1}{2}  \|b^\eta\|_\infty  + \dfrac{1}{p+1} (T-t)^{p-1} \gamma^\star \right]  \leq  \left[ \dfrac{1}{2}  \|b^\eta\|_\infty  + \dfrac{1}{p+1} \gamma^\star + \dfrac{R}{2} \right] (T-t)^2 \\
& \quad \leq R (T-t)^2
\end{align*}
if $\delta \leq 1$ and $R =  \|b^\eta\|_\infty  + \dfrac{2}{p+1} \gamma^\star $. To summarize, if 
$$R =  \|b^\eta\|_\infty  + \dfrac{2}{p+1} \gamma^\star ,\quad L= \dfrac{qR}{\eta_\star} 2^{|q-2|},\quad \delta = \min \left( 1, T, \dfrac{1}{2L}, \dfrac{\eta_\star}{2R}\right)$$
then $\Gamma$ is a contraction from the ball of $\mathcal H^\delta$ with radius $R$ into itself, thus has a unique fixed point $H$, which is the solution of the wanted BSDE. Moreover the solution $(H,Z^H)$ is the limit in $\mathcal H^\delta$ of the sequence $(H^k,Z^{H,k})$ unique solution in $\mathcal H^\delta$ of 
$$H^k_t = \int_t^T F(s,H^{k-1}_s) ds - \int_t^T Z^{H,k}_s dW_s$$
with $(H^0,Z^{H,0}) = (0,0)$ and for any $k$ and $t \in [T-\delta,T]$, $|H^k_t| \leq R (T-t)^2$.

Note that the generator $F$ is continuous and monotone in $h$ on $[0,T-\delta]$: 
\begin{align*}
F(t,h) - F(t,\tilde h) & = - (p-1)\eta_t \left[ \left( 1 + \dfrac{1}{\eta_t (T-t)} h\right) \left| 1 + \dfrac{1}{\eta_t (T-t)} h\right|^{q-1} - \left( 1 + \dfrac{1}{\eta_t (T-t)} \tilde h\right) \left| 1 + \dfrac{1}{\eta_t (T-t)} \tilde h\right|^{q-1} \right] \\
& \qquad +  \dfrac{p}{T-t} ( h  - \tilde h ) \\
& = - \dfrac{p}{(T-t)}(h-\tilde h) \int_0^1 \left| 1 +  \dfrac{1}{\eta_t (T-t)} \tilde h+ a  \dfrac{1}{\eta_t (T-t)} (h-\tilde h) \right|^{q-1}  da   +  \dfrac{p}{T-t} (   h  - \tilde h ),
\end{align*}
thus for $t\leq T-\delta$
$$
(F(t,h) - F(t,\tilde h))(h-\tilde h) \leq  \dfrac{p}{\delta} (   h  - \tilde h )^2. 
$$
And \begin{align*}
|F(t,h)| & \leq  (T-t) \|b^\eta\|_\infty  + (T-t)^{p} \gamma^\star  + (p-1)\eta_t \left[  \left| 1 + \dfrac{1}{\eta_t \delta} h\right|^{q} +1 + \dfrac{q}{\eta_t \delta} |h| \right] \\
& \leq T \|b^\eta\|_\infty  + T^{p} \gamma^\star + (p-1)\eta_t 2^{q-1}   \left( 1 + \dfrac{1}{\eta_t^q \delta^q} |h|^q \right)   + \dfrac{p}{ \delta} |h|  \\
& \leq  T \|b^\eta\|_\infty  + T^{p} \gamma^\star + \dfrac{p}{ \delta} |h|  + (p-1) 2^{q-1}  \dfrac{1}{\eta_\star^{q-1} \delta^q} |h|^q +  (p-1)(2^{q-1} +1)\eta_t .
\end{align*}
In particular 
$$\mathbb E \left[ \sup_{|h| \leq M} |F(t,h)|  \leq C (1 + \mathbb E \eta_t ) \right] < + \infty.$$
Since $H_{T-\delta}$ is a bounded random variable, the BSDE 
$$H_t =H_{T-\delta} +  \int_t^{T-\delta} F(s,H_s) ds - \int_t^{T-\delta} Z^H_s dW_s $$
has a unique solution on $[0,T-\delta]$ and there exists a constant $C$ such that $|H_t| \leq C$ on $[0,T-\delta]$. See \cite[Proposition 5.24]{pard:rasc:14} and \cite[Proposition 3.3]{bech:06}. Notice that the martingale $\int Z^H dW$ is a BMO-martingale on $[0,T-\delta]$. Since $H$ is bounded (by $C$), we can modify the generator $F$ outside the interval $[-C,C]$, such that $F$ is Lipschitz continuous and with linear growth w.r.t. $h$. Then we can define the sequence $H^k$ on $[0,T-\delta]$, converging to $H$ and such that a.s. $|H^k_t|\leq C$. 
\end{proof}

We define on $[0,T)$ the process 
$$\widehat Y_t = \dfrac{\eta_t}{(T-t)^{p-1}} + \dfrac{1}{(T-t)^{p}} H_t.$$
We can easily verify that for any $0\leq t \leq s < T$
$$\widehat  Y_t = \widehat Y_s + \int_t^s \left( -(p-1)\dfrac{|\widehat  Y_u|^{q-1} }{\eta_u^{q-1}}\widehat  Y_u + \gamma_u \right) du - \int_t^s \widehat Z_u dW_u,$$
and a.s. 
$$\lim_{t \to T} \widehat  Y_t  = +\infty.$$
Moreover on $[T-\delta,T]$, 
$$\widehat Y_t \geq \dfrac{\eta_t}{(T-t)^{p-1}} - R\dfrac{1}{(T-t)^{p-2}} \geq \dfrac{1}{(T-t)^{p-1}} \left( \eta_\star - R (T-t)\right)\geq \dfrac{\eta_\star}{2(T-t)^{p-1}} .$$
Thus $\widehat Y$ is non-negative on $[T-\delta, T]$. By standard comparison principle on $[0,T-\delta]$ (see \cite[Section 5.3.6]{pard:rasc:14}), $\widehat Y$ is also non-negative on $[0,T]$. Since $Y$ is the minimal non-negative solution, a.s. for any $t \in [0,T]$, $Y_t \leq \widehat  Y_t $.

From the uniqueness result of \cite[Theorem 10]{grae:popi:21}, $\widehat  Y= Y$ and thus the minimal solution of \eqref{eq:liquid_BSDE} is given by \eqref{eq:asymp_Y}. Let us evoke the main arguments. 
\begin{proposition}
If $\eta$ and the process $(\sigma^\eta_u)^2 \eta_u^{-q-1}$ are bounded, $\widehat  Y= Y$.
\end{proposition}
\begin{proof}
We split $Y$ as follows 
$$Y_t = \dfrac{\eta_t}{(T-t)^{p-1}} + \dfrac{1}{(T-t)^{p}} \mathcal H_t.$$
Since $\widehat Y \geq Y$, we deduce that a.s. for any $t$, $ \mathcal H_t \leq H_t$. Our goal is to prove that $\mathcal H = H$, thus $\widehat Y = Y$.

Since $Y_t \geq 0$, for any $t$, $-(T-t) \eta_t \leq \mathcal H_t$, and on $[T-\delta,T]$, $\mathcal H_t \leq H_t \leq R(T-t)^2$.
Thus a.s. $\lim_{t\to T} \mathcal H_t = 0$. From the dynamics of $Y$, $\mathcal H$ solves the BSDE \eqref{eq:BSDE_H} on $[0,\tau]$ for any $\tau < T$. Finally
$$1+  \dfrac{\mathcal H_t}{\eta_t (T-t)} = (T-t)^{p-1} \dfrac{Y_t }{\eta_t} \geq 0.$$
Hence
\begin{align*}
F(t,\mathcal H_t) & = \left[ (T-t) b^\eta_t  + (T-t)^{p} \gamma_t \right] - (p-1)\eta_t \left[ \left( 1 + \dfrac{1}{\eta_t (T-t)} \mathcal H_t\right) \left| 1 + \dfrac{1}{\eta_t (T-t)} \mathcal H_t\right|^{q-1} -1 - q\dfrac{1}{\eta_t (T-t)} \mathcal H_t \right] 
\end{align*}
is controlled on $[0,T]$: 
$$
- \eta_t  - (p-1)\eta_t \left( 1 + \dfrac{R}{\eta_t }(T-t) \right)^{q}   \leq F(t,\mathcal H_t) - \left[ (T-t) b^\eta_t  + (T-t)^{p} \gamma_t \right] \leq (p-1)\eta_t + pR(T-t) . 
$$
Hence $\mathcal H$ is also a solution of the BSDE \eqref{eq:BSDE_H}. 

Now let us consider $\Delta H=  H - \mathcal H$, $\Delta Z=  Z^H - \mathcal  Z$ and we proceed as in the proof of \cite[Proposition 20]{grae:popi:21}. We denote 
$$g(y) =  |y|^{q-1} y, \quad g'(y) = q |y|^{q-1},\quad g''(y) = q(q-1) |y|^{q-2} \mbox{sgn} (y).$$
Then for $0 \leq t \leq T$
\begin{align*}
\Delta H_t & = \int_t^T F(s, H_s)  - F(s,\mathcal H_s) ds - \int_t^T \Delta Z_s dW_s \\
& = -(p-1)\int_t^T \eta_s \left[ g \left( 1 + \dfrac{H_s}{\eta_s (T-s)}  \right) - g \left( 1 + \dfrac{\mathcal  H_s}{\eta_s (T-s)} \right) -g' \left( 1 + \dfrac{\mathcal H_s}{\eta_s (T-s)} \right) \dfrac{\Delta H_s}{\eta_s(T-s)}    \right] ds \\
& \qquad  -(p-1)\int_t^T \eta_s \left[ g' \left( 1 + \dfrac{\mathcal H_s}{\eta_s (T-s)} \right) -q \right]  \dfrac{\Delta H_s}{\eta_s(T-s)}  ds  - \int_t^T \Delta Z_s dW_s .
\end{align*}
Define for $T-\delta \leq t \leq s < T$
\begin{align*}
\Xi_s & = \eta_s \left[ g \left( 1 + \dfrac{H_s}{\eta_s (T-s)}  \right) - g \left( 1 + \dfrac{\mathcal H_s}{\eta_s (T-s)} \right) -g' \left( 1 + \dfrac{\mathcal H_s}{\eta_s (T-s)} \right) \dfrac{\Delta H_s}{\eta_s(T-s)}    \right] , \\
\Upsilon_{t,s} & = \int_t^s \eta_u \left[ g' \left( 1 + \dfrac{\mathcal H_u}{\eta_u (T-u)} \right) -q \right]  \dfrac{1}{\eta_u(T-u)} du = \int_t^s q \left[ \left| 1 + \dfrac{\mathcal H_u}{\eta_u (T-u)} \right|^{q-1} -1 \right]  \dfrac{1}{(T-u)} du . 
\end{align*}
Then for any $\varepsilon > 0$
\begin{align*}
\Delta H_t & =  \mathbb E^{\mathcal{F}_t} \left[ \Delta H_{T-\varepsilon} \exp(-(p-1) \Upsilon_{t,T-\varepsilon} )  \right] -(p-1) \mathbb E^{\mathcal{F}_t} \left[ \int_t^{T-\varepsilon} \Xi_s \exp(-(p-1) \Upsilon_{t,s} ) ds \right] \\
& \leq \mathbb E^{\mathcal{F}_t} \left[ \Delta H_{T-\varepsilon} \exp(-(p-1) \Upsilon_{t,T-\varepsilon} ) \right] 
\end{align*}
since $\Xi_s \geq 0$ because $g$ is a convex function on $[0,\infty)$. 

Let us explain how to control the negative part of $\Upsilon_{t,s}$. From \cite[Lemma 6]{grae:popi:21}, we also know that a.s. for any $t \in [0,T]$
$$Y_t \geq \left[ \mathbb E^{\mathcal{F}_t} \left( \int_t^T \eta_s^{1-q} ds \right) \right]^{1-p}.$$
Thus 
$$1+  \dfrac{\mathcal H_t}{\eta_t (T-t)} = (T-t)^{p-1} \dfrac{Y_t}{\eta_t}\geq \left[ \dfrac{\eta_t^{q-1}}{T-t} \mathbb E^{\mathcal{F}_t} \left( \int_t^T \eta_s^{1-q} ds \right) \right]^{1-p}  \geq \dfrac{\eta_\star}{\eta_t}.$$
Evoke that $\eta$ is an It\^o process, bounded from below by $\eta_\star > 0$. By It\^o's formula for $t \leq s \leq T$
$$ \eta_s^{1-q} = \eta_t^{1-q} + \int_t^s (1-q) \eta_u^{-q} b^\eta_u du + \int_t^s (1-q) \eta_u^{-q} \sigma^\eta_u dW_u + \dfrac{q(q-1)}{2} \int_t^s  \eta_u^{-q-1} |\sigma^\eta_u|^2 du.$$
Hence
\begin{align*}
\mathbb E^{\mathcal{F}_t} \left( \int_t^T \eta_s^{1-q} ds \right) & = (T-t) \eta_t^{1-q} +  (q-1) \mathbb E^{\mathcal{F}_t} \left[ \int_t^T  \left( \int_t^s  (- \eta_u^{-q} b^\eta_u +  \dfrac{q}{2}  \eta_u^{-q-1} |\sigma^\eta_u|^2 ) du \right) ds \right] \\
& = (T-t) \eta_t^{1-q} +  (q-1) \mathbb E^{\mathcal{F}_t} \left[ \int_t^T     (T-u)\theta_u du \right] 
\end{align*}
with 
$$\theta_u = - \eta_u^{-q} b^\eta_u +  \dfrac{q}{2}  \eta_u^{-q-1} |\sigma^\eta_u|^2 = \dfrac{1}{\eta_u^{q-1}} \left(  \dfrac{q}{2} \dfrac{ |\sigma^\eta_u|^2}{\eta_u^2}    -\dfrac{b^\eta_u}{\eta_u} \right)  .$$
Thus
\begin{align*}
Y_t & \geq \left[ (T-t) \eta_t^{1-q} +  (q-1) \mathbb E^{\mathcal{F}_t} \left( \int_t^T (T-u)\theta_u du \right)  \right]^{1-p} \\
& \geq \dfrac{\eta_t}{(T-t)^{p-1}}  \left[ 1 +  
\dfrac{ (q-1) \eta_t^{q-1}}{T-t} \mathbb E^{\mathcal{F}_t} \left( \int_t^T (T-u)\theta_u du \right)  \right]^{1-p} ,
\end{align*}
and
\begin{align*}
(T-t)^{p-1} \dfrac{Y_t}{\eta_t}  & = 1 +  \dfrac{\mathcal H_t}{\eta_t (T-t)} \geq \left[ 1 +  
\dfrac{ (q-1) \eta_t^{q-1}}{T-t} \mathbb E^{\mathcal{F}_t} \left( \int_t^T     (T-u)\theta_u du \right)  \right]^{1-p}  .
\end{align*}
We deduce that 
\begin{align*}
& \left|1+  \dfrac{\mathcal H_t}{\eta_t (T-t)} \right|^{q-1} -1  \geq \left[ 1 +  
\dfrac{ (q-1) \eta_t^{q-1}}{T-t} \mathbb E^{\mathcal{F}_t} \left( \int_t^T     (T-u)\theta_u du \right)  \right]^{-1}  -1 \\
&\quad =- \left[ 1 +  
\dfrac{ (q-1) \eta_t^{q-1}}{T-t} \mathbb E^{\mathcal{F}_t} \left( \int_t^T     (T-u)\theta_u du \right)  \right]^{-1} \dfrac{ (q-1) \eta_t^{q-1}}{T-t} \mathbb E^{\mathcal{F}_t} \left( \int_t^T     (T-u)\theta_u du \right) \\
& \quad \geq -  \dfrac{ (q-1) \eta_t^{q-1}}{T-t}\mathbb E^{\mathcal{F}_t} \left( \int_t^T (T-u)\theta_u du   \right) .
\end{align*}
Hence 
\begin{align*}
(p-1)\Upsilon_{t,s} & = (p-1) \int_t^s q \left[ \left| 1 + \dfrac{\mathcal H_u}{\eta_u (T-u)} \right|^{q-1} -1 \right]  \dfrac{1}{(T-u)} du \geq -  q \int_t^s \dfrac{  \eta_u^{q-1}}{(T-u)^2}\mathbb E^{\mathcal{F}_u} \left( \int_u^T (T-r)\theta_r dr \right) du . 
\end{align*}
If $\eta$ and $\theta$ are bounded by $\eta^\star$ and $\|\theta\|_\infty$, we obtain
\begin{align*}
(p-1)\Upsilon_{t,s} &  \geq -  \frac{q}{2} (\eta^\star)^{q-1} \|\theta\|_\infty . 
\end{align*}
Since for any $\varepsilon > 0$,
\begin{align*}
\Delta H_t &  \leq \mathbb E^{\mathcal{F}_t} \left[ \Delta H_{T-\varepsilon} \exp(-(p-1) \Upsilon_{t,T-\varepsilon} )  \right] \leq  C \mathbb E^{\mathcal{F}_t} \left[ \Delta H_{T-\varepsilon} \right]
\end{align*}
we can pass to the limit to deduce that $\Delta H \leq 0$, that is $H\leq \mathcal H$. 
 \end{proof}

\bibliographystyle{abbrv}
\bibliography{biblio}

\end{document}